\documentclass[12pt]{amsart}
\usepackage{amssymb,amscd,amsthm,amsmath,color}
\usepackage{fullpage}
\usepackage{epsfig}
\usepackage{graphicx}
\usepackage{xypic}
\usepackage{url}
\usepackage{color}
\usepackage{hyperref}
\usepackage{enumerate}
\usepackage{enumitem}
\usepackage{comment}
\usepackage{mathtools}
\usepackage{quiver}
\usepackage{tikz, tikz-cd}
\usepackage{verbatim}

\numberwithin{equation}{section}

\usepackage[
backend=biber,
style=alphabetic,
sorting=nyt
]{biblatex}
\newtheorem{theorem}{Theorem}[section]

\newtheorem{lemma}[theorem]{Lemma}

\newtheorem{corollary}[theorem]{Corollary}
\newtheorem{proposition}[theorem]{Proposition}

\theoremstyle{definition}

\newtheorem{remark}[theorem]{Remark}
\newtheorem{example}[theorem]{Example}

\newcommand{\cI}{\mathcal{I}}

\newcommand{\cO}{\mathcal{O}}

\renewcommand{\P}{\mathbb{P}}

\DeclareMathOperator{\Hom}{Hom}

\DeclareMathOperator{\rk}{rk}

\DeclareMathOperator{\coker}{coker}

\DeclareMathOperator{\Ext}{Ext}

\DeclareMathOperator{\sHom}{\mathcal{H}\kern -.5pt\mathit{om}}
\DeclareMathOperator{\sTor}{\mathcal{T}\kern -1.5pt\mathit{or}}

\DeclareMathOperator{\OG}{OG}

\DeclareMathOperator{\sEnd}{\mathcal{E}\kern -.5pt\mathit{nd}}
\DeclareMathOperator{\sExt}{\mathcal{E}\kern -.5pt\mathit{xt}}

\begin{document}

\title{Restricted tangent bundle of rational curves on projective hypersurfaces}
\author{Lucas Mioranci}
\address{Instituto de Matemática Pura e Aplicada, Estrada Dona Castorina 110, Jardim Botânico, 22460-320 Rio de Janeiro, RJ, Brazil}
\email{lucas.mioranci@impa.br}
\date{}
\thanks{During the preparation of this article, the author was supported by CNPq Post-doctoral fellowship 152039/2024-4}

\begin{abstract}
    We determine all triples $(e,d,n)$ for which a general degree $d$ hypersurface $X\subset \P^n$ contains a degree $e$ rational curve $C$ with balanced restricted tangent bundle $T_X|_C$. In addition, we show how to compute explicit examples of hypersurfaces with balanced $T_X|_C$ when $C$ is a rational normal curve.
\end{abstract}

\maketitle

\section{Introduction}

The normal and restricted tangent bundles of a curve on a variety give fundamental information on the local structure of the space of its deformations. They tell the dimension of the space of curves and the freedom we have in deforming the curve on the variety. In a sense, curves with a ``more balanced" normal and restricted tangent bundles are the ``most free" ones and the ones whose deformations interpolate the maximum number of points. In this paper, we show when general projective hypersurfaces of degree $d$ in $\P^n$ have degree $e$ rational curves with balanced restricted tangent bundle. We work over an algebraically closed field $k$ of characteristic $p$ that does not divide the degree $e$ of the curve.

By the Birkhoff-Grothendieck theorem, a vector bundle $E$ on $\P^1$ splits as a direct sum of line bundles, $E = \bigoplus_{i=1}^{r}\cO_{\P^1}(a_i)$ for integers $a_1\le \cdots \le a_r$. The collection $\{a_i\}$ is called the \textit{splitting type} of $E$. The vector bundle is called \textit{balanced} if $|a_i-a_j|\le 1$ for all $1\le i,j\le r$, and \textit{perfectly balanced} if all $a_i$ are equal. We remark that there exists a unique balanced splitting type for vector bundles of a given rank and degree. Also, being balanced is an open condition in a family of vector bundles (see Section \ref{section_vector_bundles}).

Let $X$ be a smooth degree $d$ hypersurface in $\P^n$ containing a smooth rational curve $C$ of degree $e$. There has been great interest in describing the possible splitting types of the normal bundle $N_{C/X}$ and the restricted tangent bundle $T_X|_C$. For $X=\P^n$, there is a long history of works describing the space of curves having a fixed normal or restricted tangent bundle, see \cite{GS80,Sa80,Sa82,EV81,EV82,M86,Ascenzi_restricted_tangent_in_P2,Ran,Gimigliano_plane_rational_tangent_bundle,AlzatiRe,AlzatiReTortora,CR18,Ascenzi_restricted_tangent_in_P3,Larson_Vogt_Interpolation_Brill_Noether_curves}. A fair amount of work is done for rational curves and their interpolation properties in more general varieties $X$, especially projective complete intersections and Grassmannians, see \cite{Kollar_rational_curves_on_algebraic_varieties,AlzatiRe_PGL2_actions_Grassmannians,Fur16,L21,Interpolation_rational_scrolls,Ran_low_degree,Ran_balanced_curves,Ran_Regular_and_rigid_on_Calabi_Yau}. These questions can be generalized to higher genus curves, and have been studied in $\P^n$ in \cite{Ellingsrud_and_Laksov_elliptic_space_curves,Hulek_Projective_geometry_of_elliptic_curves,Perrin_Courbes_Passant,Ein_Lazarsfeld_Stability_and_restrictions_of_Picard_bundles,Hein_Kurke_restricted_tangent_bundle_space_curves,Hein_Curves_P3_with_good_restriction_tangent_bundle,E_Larson_interpolation_restricted_tangent_general_curves,Atanasov_Larson_Yang_Interpolation_for_normal_bundles}, for Fano hypersurfaces in \cite{Interpolation_of_curves_on_Fano_hypersurfaces} and for Grassmannians in \cite{Ballico_Ramella_restricted_tangent_in_Grassmannians,Coskun_Larson_Vogt_normal_bundles_Grassmannians}.

In \cite{CR}, Coskun and Riedl show that a general Fano hypersurface of degree $d\ge 2$ in $\P^n$ contains rational curves of degree $e$ with a balanced normal bundle for every degree $1\le e\le n$. This result was later extended by Ran \cite{Ran_low_degree} for degrees $1\le e\le 2n-2$ and $d\ge 4$. A description of all possible splitting types for the normal bundle when $C$ is a rational normal curve and $X$ is a projective hypersurface, in addition to the space of hypersurfaces inducing each splitting type, has been done for lines in \cite{L21} and higher-degree curves in \cite{Mioranci_Normal_Bundle}.

For the restricted tangent bundle, however, much less is known. In \cite{Interpolation_of_curves_on_Fano_hypersurfaces}, Ran, working in arbitrary genus and Fano projective hypersurfaces, shows the existence of curves with balanced restricted tangent bundles for large degrees $e$ in some arithmetic progressions, and conjectures the existence of obstructions in terms of degree and genus for the existence of such curves. He highlights the \textit{modular interpolation} property of the tangent bundle: for rational curves $f:\P^1\to X$, a restricted tangent bundle $f^*T_X\cong \bigoplus_{i=1}^{n-1}\cO(a_i)$ with $a_1\le \cdots \le a_{n-1}$ means that, for $a_1+1$ general points $p_i\in \P^1$ and general points $x_i\in X$, there are deformations $\tilde{f}$ of $f$ such that $\tilde {f}(p_i)=x_i$. For curves $f$ with fixed degree $e$, the expected (and maximum) number of points that can be interpolated as above is achieved when $f^*T_X$ is balanced. In this case, we can interpolate up to $\lfloor \frac{\deg f^*K_X}{n-1} \rfloor + 1 = \lfloor \frac{e(n+1-d)}{n-1} \rfloor + 1$ points (see Section \ref{Section_modular_interpolation}).

In this paper, we determine the existence of degree $e$ rational curves with balanced restricted tangent bundle in general Fano hypersurfaces for every degree $e$. We remark that the restricted tangent bundle is never balanced for non-Fano hypersurfaces (see Proposition \ref{tangent_degree_less_than_2}).

\begin{theorem}(see Proposition \ref{tangent_degree_less_than_2} and Theorem \ref{theorem_higher_degree_curves})
    Let $X\subset \P^n$ be a smooth Fano hypersurface of degree $d$, $3\le d\le n$.
    \begin{enumerate}[itemsep = 5pt]
        \item The restricted tangent bundle $T_X|_C$ of a degree $e\le \frac{n-1}{n+1-d}$ rational curve $C$ on $X$ is never balanced.

        \item A general hypersurface $X$ contains rational curves of degree $e$ with balanced restricted tangent bundle for every $e > \frac{n-1}{n+1-d}$.
    \end{enumerate}
\end{theorem}

The quadrics ($d=2$) form a very special case in which only even-degree curves can have a balanced restricted tangent bundle. Odd-degree curves may have restricted tangent bundles as close as possible to the balanced splitting type, but cannot be balanced. We use a ruled surface construction from \cite{Kollar_Quadratic_solutions} that relates pairs of curves of different degrees in quadric hypersurfaces, and then show that odd-degree curves always interpolate fewer than the expected points.

\begin{theorem}(see Theorem \ref{theorem_odd_degree_curves_on_quadrics_are_not_balanced} and Theorem \ref{Corollary_quadrics})
    Let $X\subset \P^n$, $n\ge 3$, be a smooth quadric hypersurface.
    \begin{enumerate}
        \item For every even $e\ge 2$, $X$ contains degree $e$ rational curves with balanced restricted tangent bundle $T_X|_C\cong \cO(e)^{n-1}$.

        \item No odd-degree rational curve on $X$ has balanced restricted tangent bundle. For every odd $e\ge 1$, there exist degree $e$ rational curves with $T_X|_C\cong \cO(e-1)\oplus \cO(e)^{n-3}\oplus \cO(e+1)$.
    \end{enumerate}
\end{theorem}

We approach the problem by considering the particular case of hypersurfaces $X$ containing a degree $e$ rational normal curve $C$. In this case, the restricted tangent bundle is the kernel in the short exact sequence
\[ 0\longrightarrow T_X|_C\longrightarrow \cO(e+1)^e\oplus \cO(e)^{n-e}\overset{\delta}{\longrightarrow}\cO(de)\longrightarrow 0, \]
where $T_{\P^n}|_C\cong \cO(e+1)^e\oplus \cO(e)^{n-e}$ and $N_{X/\P^n}|_C\cong \cO(de)$. By choosing the appropriate hypersurface $X$, we can produce a map $\delta$ inducing a balanced kernel. This allows us to produce explicit examples of hypersurfaces $X$ with balanced restricted tangent bundle.

\begin{theorem}(see Theorems \ref{theorem_quadrics}, \ref{theorem_cubics}, \ref{theorem_quartics} and \ref{theorem_higher_degree})
    Let $X$ be a general degree $2\le d\le 4$ hypersurface containing a rational normal curve $C$ of degree $e\le n$. We list the exact splitting type of $T_X|_C$ for every $e$ and $n$ and compute explicit examples of $X$ for each one. We also compute explicit examples of degree $d$ hypersurfaces $X$ with balanced $T_X|_C$ when $n\ge 2d-2$ and $C$ is the rational normal curve of degree $n$ in $\P^n$.
\end{theorem}

The examples can be worked out for degrees higher than 4, although the computations get more involved. We invite the reader to try the examples for particular cases of $(d,e,n)$ in \textit{Macaulay2} \cite{Macaulay2}.

One of the main tools in our proofs is Proposition \ref{Ext_Proposition}, which implies that for $C$ a rational normal curve, if $T_Y|_C$ is balanced for a degree $d$ hypersurface $Y\subset \P^{n-1}$, then we can extend $Y$ to a degree $d$ hypersurface $X\subset \P^n$ with $T_X|_C$ also balanced. This allows us to construct the case $e=n$ and work out inductively when $n>e$. The induction is done constructively for $d\le 4$, so we can produce the explicit examples of hypersurfaces. We also take advantage of the induction when $T_X|_C$ splits as $N_{C/X}\oplus \cO(2)$ with $N_{C/X}$ balanced. The known results for normal bundles and induction give us the restricted tangent bundle for curves of degree $e\le \max\{n,2d-2\}$. We then glue curves to obtain rational curves with balanced restricted tangent bundle for the remaining degrees $e$.

\subsection*{Organization of the paper}

Section 2 is dedicated to reviewing some preliminary results and describing the computation of the map $\delta$. We also introduce some notions on curve interpolation and specialization of vector bundles. Section 3 develops the induction argument and applies it to curves of degree $1\le e\le 2d-2$. In Section 4, we prove the theorem for quadrics. Sections 4, 5, and 6 show the computation of examples of hypersurfaces of degree $d=2,3,4$, respectively. In Section 7, we obtain examples for $2d-2\le e\le n$ and glue rational curves to obtain higher-degree curves with balanced restricted tangent bundle.

\subsection*{Acknowledgments}

I am very grateful to Carolina Araujo, Eduardo Esteves, Ziv Ran, Izzet Coskun, and Eric Riedl for discussions and correspondence on rational curves and their restricted tangent bundles. I thank Coskun and Riedl for suggesting the degeneration approach for higher-degree rational curves.

\section{Preliminaries}

\subsection{Vector bundles on rational curves}\label{section_vector_bundles}

Vector bundles on rational curves can only get ``more balanced" under generalization.

\begin{lemma}\label{lemma_balanced_is_open_condition}\cite[Theorem 14.7(a)]{3264}
    Let $E_1$ and $E_2$ be two vector bundles on $\P^1$ of same degree $d$ and rank $r$. Write their decomposition as direct sums of line bundles as
    \[ E_1 = \bigoplus_{i=1}^r \cO(a_i) \ \text{ and } \ E_2 = \bigoplus_{j=1}^r \cO(b_j), \]
    with $\{a_i\}$ and $\{b_j\}$ non-decreasing sequences. The vector bundle $E_1$ specializes to $E_2$ if and only if for every integer $k$ satisfying $1\le k\le r$, we have
    \[ \sum_{i=1}^{k}a_i\ge \sum_{j=1}^{k}b_j. \]
\end{lemma}

In particular, being balanced is an open condition in a family of vector bundles on a rational curve. Thus, if we can find specific examples of hypersurfaces and rational curves for which the restricted tangent bundle is balanced, then we can conclude that the balancedness is maintained for the general member in their families.

For $E$ a vector bundle in $\P^1$, we denote by $\mu(E)$ the \textit{slope} of $E$, defined by
\[ \mu(E) = \frac{\deg E}{\rk E}. \]
If $E$ is balanced, then $E$ has a unique decomposition as a sum of line bundles $\cO(\left \lfloor \mu(E) \right \rfloor )$ and $\cO(\left \lceil \mu(E) \right \rceil )$.

\subsection{Rational normal curves}

Let $e\le n$. We define the \textit{rational normal curve} of degree $e$ in $\P^n$ as the curve $C$ defined by
\[ f = \left ( s^e : s^{e-1}t : s^{e-2}t^2 : \cdots : st^{e-1} : t^e : 0 : \cdots : 0 \right ): \P^1\to \P^n. \]

Any projective change of coordinates of $C$ is often also called a rational normal curve of degree $e$, but to fix coordinates, we will refer to it as the curve defined above.

Observe that $C$ spans the linear space $\P^e = V(x_{e+1}, \ldots , x_n)$. Define the quadratic forms $Q_{i,j} = x_ix_{j-1} - x_{i-1}x_j$ for $1\le i<k\le e$, which correspond to the $2\times 2$ minors of the matrix
\[ \begin{bmatrix}
x_1 & x_2 & \cdots & x_e \\
x_0 & x_1 & \cdots & x_{e-1} 
\end{bmatrix}. \]

Together with the linear forms cutting out $\P^e$, they generate the homogeneous ideal $I_{C}\subset k[x_0, \ldots , x_n]$ of the rational normal curve:
\[ I_{C} = \left ( \{ Q_{i,j} \ | \ 1\le i<j\le e \} \cup \{x_{e+1}, \ldots , x_n\} \right ). \]

In \cite[Proposition 2.4]{CR}, Coskun and Riedl use the relations between the generators of $I_C$ to show that the quadrics $Q_{i,i+1}$, for $1\le i\le n-1$, suffice to determine $N_{C/\P^n}$.

\begin{proposition}\label{relations}(\cite{CR} Proposition 2.4)
Let $C$ be the rational normal curve of degree $n$ in $\P^n$. An element $\alpha \in H^0(N_{C/\P^n}) = \Hom (\cI_{C/\P^n}, \cO_C)$ is determined by the images $\alpha (Q_{i,i+1})$, for $1\le i\le n-1$. Furthermore, $s^{n-i-1}t^{i-1}$ divides $\alpha (Q_{i,i+1})$ and this is the only constraint on $\alpha (Q_{i,i+1})$. If $b_{i,i+1}$, for $1\le i\le n-1$, are arbitrary polynomials of degree $n+2$, there exists an element $\alpha \in H^0(N_{C/\P^n})$ such that $\alpha (Q_{i,i+1}) = s^{n-i-1}t^{i-1}b_{i,i+1}$.

In addition, the image $\alpha (Q_{i,j})$ of the other generators of $I_{C}$ are expressed in terms of $b_{l,l+1}$ by
\[ \alpha (Q_{i,j}) = \sum _{l=i}^{j-1} s^{n-j-i+l}t^{j+i-l-2}b_{l,l+1}. \]
\end{proposition}

\begin{corollary}\label{normalbundlee} (\cite{CR} Corollary 2.6)
Let $C$ be the degree $e$ rational normal curve in $\P^n$. Then the normal bundle $N_{C/\P^n}$ is $N_{C/\P^e}\oplus N_{\P^e/\P^n}\cong \cO_{\P^1}(e+2)^{e-1}\oplus \cO_{\P^1}(e)^{n-e}$.
\end{corollary}

By using the Euler sequence restricted to $C$:
\[ 0\longrightarrow \Omega_{\P^n}|_C \longrightarrow \cO_{\P^1}(-e)^{n+1} \overset{f}{\longrightarrow} \cO_{\P^1} \longrightarrow 0, \]
we can compute the restricted tangent sheaf $T_{\P^n}|_C$.

\begin{proposition}\label{tangent_in_Pn}(see \cite[Proposition 3.3]{CR18})
    Let $C$ be the rational normal curve of degree $e$ in $\P^n$. Then
    \[ T_{\P^n}|_C\cong T_{\P^e}|_C\oplus N_{\P^e/\P^n}\cong \cO_{\P^1}(e+1)^{e}\oplus \cO_{\P^1}(e)^{n-e}. \]
\end{proposition}

\begin{remark}
    We can also compute the normal bundle $N_{C/\P^n}$ by using the sequence
    \[ 0\longrightarrow N_{C/\P^n}^* \longrightarrow \cO_{\P^1}(-e)^{n+1} \overset{\partial f}{\longrightarrow} \cO_{\P^1}^2 \longrightarrow 0, \]
    where $\partial f$ is the transpose of the Jacobian matrix, as it is done in \cite[Theorem 3.2]{CR18}.
\end{remark}

\subsection{Normal bundles on hypersurfaces}

Let $X$ be a degree $d$ hypersurface in $\P^n$ containing the degree $e$ rational normal curve $C$ and smooth along $C$. There is a short exact sequence of normal bundles:
\[ 0\longrightarrow N_{C/X}\longrightarrow N_{C/\P^n}\overset{\psi}{\longrightarrow} N_{X/\P^n|_C}\longrightarrow 0. \]

By the identification $N_{X/\P^n}\cong \cO_X(d)$ and Corollary \ref{normalbundlee}, this sequence is equivalent to
\[ 0\longrightarrow N_{C/X}\longrightarrow \cO(e+2)^{e-1}\oplus \cO(e)^{n-e}\overset{\psi}{\longrightarrow}\cO(de)\longrightarrow 0. \]

In particular, every hypersurface $X$ defined by a degree $d$ polynomial $F\in H^0(\cI_{C/\P^n}(d))$ induces a map $\psi$ of normal bundles, thus defining a map:
\[ \phi: H^0(\cI_{C/\P^n}(d))\to \Hom (\cO(e+2)^{e-1}\oplus \cO(e)^{n-e}, \cO(de)). \]

Proposition \ref{relations} allows us to explicitly obtain the map $\psi_F$ for every given polynomial $F$. First, let $e=n$. Write $F$ as a combination of the generators of $I_C$, $F = \sum_{1\le i<j\le n} F_{i,j}Q_{i,j}$. Then $\psi_F(\alpha) = \sum_{1\le i<j\le n} F_{i,j}|_C\cdot \alpha (Q_{i,j})$. By the relations from Proposition \ref{relations}, we have
\[ \psi_F(\alpha) = \sum_{1\le i<j\le n} F_{i,j}|_C\sum_{l=i}^{j-1} s^{n-j-i+l}t^{j+i-l-2}b_{l,l+1}. \]

Collect the terms and write the sum as $\sum_{i=1}^{n-1}C_ib_{i,i+1}$, then the map $\psi_F: \cO(n+2)^{n-1}\to \cO(dn)$ is given by the matrix $(C_1, \cdots , C_{e-1})$.

For $e<n$, the normal bundle $N_{C/\P^n}$ splits as the direct sum $N_{C/\P^e}\oplus N_{\P^e/\P^n}$. So, if we write $F$ as
\[ F = \sum_{1\le i<j\le n} F_{i,j}Q_{i,j} + \sum_{k=e+1}^n G_kx_k, \]
and collect the coefficients $C_1, \ldots , C_{e-1}$ of the $b_{l,l+1}$ as above, then the map $\psi_F$ is given by the matrix
\[ \psi_F = \left (C_1, \cdots , C_{e-1}; G_{e+1}|_C, \cdots , G_n|_C\right ): \cO(e+2)^{e-1}\oplus \cO(e)^{n-e}\to \cO(de). \]

\begin{proposition}\label{surjection_of_phi}(\cite[Theorem 3.1]{CR})
    If $d\ge 3$, then the homomorphism $\phi$ is surjective, that is, every map $\psi \in \Hom (\cO(e+2)^{e-1}\oplus \cO(e)^{n-e}, \cO(de))$ is induced by some hypersurface $X$.
\end{proposition}

It is useful to remark that we can get the $F_{i,j}|_C$ and $G_k|_C$ to be any polynomial in $s,t$ of the corresponding degree, since rational normal curves are projectively normal:

\begin{lemma}\label{projectivelynormal}\cite{ACGH}
For every $k\ge 1$, the map $H^0(\cO_{\P^n}(k))\to H^0(\cO_{C}(k))\cong H^0(\cO_{\P^1}(ek))$, $F\mapsto F|_C$, is surjective.
\end{lemma}

Given a polynomial $F|_C\in H^0(\cO_{\P^1}(ek))$, we can easily find an $F$ that restricts to it: write each monomial of $F|_C$ as a product of $k$ monomials of degree $e$. For instance, for $F|_C = s^{ek-2}t^2$ we can write $s^{ek-2}t^2 = s^{e(k-1)}(s^{e-2}t^2)$ and choose $F = x_0^{k-1}x_2$.

\subsection{Restricted tangent bundles of hypersurfaces}

Let $X$ be a degree $d$ hypersurface in $\P^n$ containing the degree $e$ rational normal curve $C$. Say $X$ is defined by a homogeneous polynomial $F$ of degree $d$. We can see the restricted tangent bundle $T_X|_C$ as the kernel of the standard tangent bundle sequence:
\[ 0\longrightarrow T_X|_C\longrightarrow T_{\P^n}|_C\overset{\delta}{\longrightarrow} N_{X/\P^n|_C}\longrightarrow 0. \]

By Proposition \ref{tangent_in_Pn}, this sequence can be written as
\[ 0\longrightarrow T_X|_C\longrightarrow \cO(e+1)^e\oplus \cO(e)^{n-e}\overset{\delta}{\longrightarrow} \cO(de)\longrightarrow 0. \]

By combining the Euler sequence of $\P^n$ restricted to $C$ and the tangent bundle sequence, we can see the map $\delta$ above as the quotient of the gradient of $F$, $\nabla F = \left(\frac{\partial F}{\partial x_0}, \cdots , \frac{\partial F}{\partial x_n}\right )$:
\[\begin{tikzcd}
	&& 0 \\
	&& {\mathcal{O}_{\mathbb{P}^1}} \\
	&& {\mathcal{O}_{\mathbb{P}^n}(e)^{n+1}} & {N_{X/\mathbb{P}^n|_C}} & 0 \\
	0 & {T_X|_C} & {T_{\mathbb{P}^n}|_C} & {N_{X/\mathbb{P}^n|_C}} & 0 \\
	&& 0
	\arrow[from=1-3, to=2-3]
	\arrow["f", from=2-3, to=3-3]
	\arrow["{\nabla F}", from=3-3, to=3-4]
	\arrow[from=3-3, to=4-3]
	\arrow[from=3-4, to=3-5]
	\arrow[equals, from=3-4, to=4-4]
	\arrow[from=4-1, to=4-2]
	\arrow[from=4-2, to=4-3]
	\arrow["\delta", from=4-3, to=4-4]
	\arrow[from=4-3, to=5-3]
	\arrow[from=4-4, to=4-5]
\end{tikzcd}\]

Alternatively, and due to Proposition \ref{surjection_of_phi}, it can be very useful to describe $\delta$ in terms of the map of normal bundles $\psi$. We first describe the maps of the tangent bundle in $\P^n$ by the following commutative diagram, whose rows are the Euler sequences for $\P^1$ and $\P^n$ (see \cite{EV81, GS80}). For this sequence, we need to assume $p \nmid  e$.
\[\begin{tikzcd}
	&& 0 & 0 \\
	0 & {\mathcal{O}_{\mathbb{P}^1}} & {\mathcal{O}_{\mathbb{P}^1}(1)^2} & {T_{\mathbb{P}^1}} & 0 \\
	0 & {\mathcal{O}_{\mathbb{P}^1}} & {\mathcal{O}_{\mathbb{P}^1}(e)^{n+1}} & {T_{\mathbb{P}^n|_C}\cong \cO(e+1)^{e}\oplus \cO(e)^{n-e}} & 0 \\
	&& {N_{C/\mathbb{P}^n}} & {N_{C/\mathbb{P}^n}\cong \cO(e+2)^{e-1}\oplus \cO(e)^{n-e}} \\
	&& 0 & 0
	\arrow[from=1-3, to=2-3]
	\arrow[from=1-4, to=2-4]
	\arrow[from=2-1, to=2-2]
	\arrow["{\binom{s}{t}}", from=2-2, to=2-3]
	\arrow["e"', from=2-2, to=3-2]
	\arrow["{(t, -s)}", from=2-3, to=2-4]
	\arrow["Jf"', from=2-3, to=3-3]
	\arrow[from=2-4, to=2-5]
	\arrow["df"', from=2-4, to=3-4]
	\arrow[from=3-1, to=3-2]
	\arrow["f", from=3-2, to=3-3]
	\arrow[from=3-3, to=3-4]
	\arrow[from=3-3, to=4-3]
	\arrow[from=3-4, to=3-5]
	\arrow["\beta", from=3-4, to=4-4]
	\arrow[equals, from=4-3, to=4-4]
	\arrow[from=4-3, to=5-3]
	\arrow[from=4-4, to=5-4]
\end{tikzcd}\]
In the diagram, $f=(s^e, s^{e-1}t, \cdots , t^e, 0, \cdots , 0)$ is the map defining $C$, and $Jf$ is the Jacobian matrix
\[ Jf = \begin{pmatrix}
\frac{\partial f_0}{\partial s} & \frac{\partial f_0}{\partial t} \\
\vdots & \vdots \\
\frac{\partial f_n}{\partial s} & \frac{\partial f_n}{\partial t} 
\end{pmatrix} = \begin{pmatrix}
es^{e-1} & 0 \\
(e-1)s^{e-2}t & s^{e-1} \\
\vdots & \vdots \\
t^{e-1} & (e-1)st^{e-2} \\
0 & et^{e-1} \\
0 & 0 \\
\vdots & \vdots \\
0 & 0
\end{pmatrix}. \]

We can compute the cokernel of $f$ and use it to show that
\[ df = \left ( s^{e-1}, s^{e-2}t, \cdots , t^{e-1}; 0, \cdots , 0 \right ). \]

Thus, the cokernel of $df$ can be obtained, and we get the map $\beta: T_{\P^n}|_C\to N_{C/\P^n}$,
\[ \beta = \begin{pmatrix}
t & -s &  &  &  &  &  &  &  \\
 & t & -s &  &  &  &  &  &  \\
 &  &  & \ddots &  &  &  &  &  \\
 &  &  &  & t & -s &  &  &  \\
 &  &  &  &  &  & 1 &  &  \\
 &  &  &  &  &  &  & \ddots &  \\
 &  &  &  &  &  &  &  & 1 
\end{pmatrix}: \cO(e+1)^e\oplus \cO(e)^{n-e}\to \cO(e+2)^{e-1}\oplus \cO(e)^{n-e}. \]

The tangent bundle and normal bundle sequences can be related in the following commutative diagram:
\[\begin{tikzcd}
	& 0 & 0 \\
	& {T_{\mathbb{P}^1}} & {T_{\mathbb{P}^1}} \\
	0 & {T_X|_C} & {T_{\mathbb{P}^n}|_C} & {N_{X/\mathbb{P}^n}|_C} & 0 \\
	0 & {N_{C/X}} & {N_{C/\mathbb{P}^n}} & {N_{X/\mathbb{P}^n}|_C} & 0 \\
	& 0 & 0
	\arrow[from=1-2, to=2-2]
	\arrow[from=1-3, to=2-3]
	\arrow[from=2-2, to=3-2]
	\arrow["df", from=2-3, to=3-3]
	\arrow[from=3-1, to=3-2]
	\arrow[from=3-2, to=3-3]
	\arrow[from=3-2, to=4-2]
	\arrow["\delta", from=3-3, to=3-4]
	\arrow["\beta", from=3-3, to=4-3]
	\arrow[from=3-4, to=3-5]
	\arrow[equals, from=3-4, to=4-4]
	\arrow[from=4-1, to=4-2]
	\arrow[from=4-2, to=4-3]
	\arrow[from=4-2, to=5-2]
	\arrow["\psi", from=4-3, to=4-4]
	\arrow[from=4-3, to=5-3]
	\arrow[from=4-4, to=4-5]
\end{tikzcd}\]
which, by Corollary \ref{normalbundlee} and Proposition \ref{tangent_in_Pn} is the same as
\[\begin{tikzcd}
	& 0 & 0 \\
	& {\mathcal{O}(2)} & {\mathcal{O}(2)} \\
	0 & {T_X|_C} & {\mathcal{O}(e+1)^e\oplus \mathcal{O}(e)^{n-e}} & {\mathcal{O}(de)} & 0 \\
	0 & {N_{C/X}} & {\mathcal{O}(e+2)^{e-1}\oplus \mathcal{O}(e)^{n-e}} & {\mathcal{O}(de)} & 0 \\
	& 0 & 0
	\arrow[from=1-2, to=2-2]
	\arrow[from=1-3, to=2-3]
	\arrow[from=2-2, to=3-2]
	\arrow["df", from=2-3, to=3-3]
	\arrow[from=3-1, to=3-2]
	\arrow[from=3-2, to=3-3]
	\arrow[from=3-2, to=4-2]
	\arrow["\delta", from=3-3, to=3-4]
	\arrow["\beta", from=3-3, to=4-3]
	\arrow[from=3-4, to=3-5]
	\arrow[equals, from=3-4, to=4-4]
	\arrow[from=4-1, to=4-2]
	\arrow[from=4-2, to=4-3]
	\arrow[from=4-2, to=5-2]
	\arrow["\psi", from=4-3, to=4-4]
	\arrow[from=4-3, to=5-3]
	\arrow[from=4-4, to=4-5]
\end{tikzcd}\]

That allows us to write the map $\delta$ as the composition $\psi \circ \beta$. The explicit computation for $\psi$ from the previous section gives us a way of finding $\delta$ for any given polynomial $F$ defining $X$. Say we got
\[ \psi = \left (C_1, \cdots , C_{e-1}; G_{e+1}|_C, \cdots , G_n|_C\right ), \]
then we obtain
\[ \delta = \psi\circ \beta = \left ( tC_1, -sC_1+tC_2, -sC_2+tC_3, \cdots , -sC_{e-2}+tC_{e-1}, -sC_{e-1}; G_{e+1}|_C, \cdots , G_{n}|_C \right ). \]

With the matrix of $\delta$, we can use its column relations to explicitly compute the desired restricted tangent bundle $T_X|_C$.

\begin{example}\label{example_n3e3d5}
    Let $n=e=3$ and $d=5$. Let $X$ be the surface defined by $F = x_0^3Q_{1,2} + x_3^3Q_{2,3}$. It induces the map on normal bundles $\psi$ given by the matrix
    \[ \psi = \left ( s^{10}, t^{10} \right ): \cO(5)^2\to \cO(15). \]

    Notice the columns $C_1$ and $C_2$ of $\psi$ above satisfy the relation
    \[ t^{10}\cdot C_1 - s^{10}\cdot C_2 = 0, \]
    which can be used to define an injective map
    \[ \cO(-5)\overset{\begin{psmallmatrix}
        t^{10} \\ -s^{10} 
        \end{psmallmatrix}}{\longrightarrow} \cO(5)^2 \]
    that factors through the kernel of $\psi$. Therefore, we can conclude that it gives the kernel of $\psi$, that is, $N_{C/X}\cong \cO(-5)$.

    With $\psi$ on hands, we can get $\delta: \cO(4)^3\to \cO(15)$, given by
    \[ \delta = \psi\circ \beta = ( s^{10}, t^{10})\cdot \begin{pmatrix}
        t & -s &  \\
         & t & -s 
    \end{pmatrix} = \left ( s^{10}t, -s^{11}+t^{11}, -st^{10} \right ). \]

    The three columns of $\delta$ satisfy the relations
    \begin{itemize}
        \item $s^2\cdot C_1 + st\cdot C_2 + t^2\cdot C_3 = 0$, and
        \item $t^9\cdot C_1 + 0\cdot C_2 + s^9\cdot C_3 = 0$.
    \end{itemize}

    And these relations define the map
    \[ K: \begin{pmatrix}
        s^2 & t^9 \\
        st & 0 \\
        t^2 & s^9 
        \end{pmatrix}: \cO(2)\oplus \cO(-5)\to \cO(4)^3, \]
    which factors through the kernel $T_X|_C$ of $\delta$. One can show $K$ is injective at all points $(s,t)\in \P^1$. Since they have the same rank and degree, we conclude that $T_X|_C\cong \cO(2)\oplus \cO(-5)$. In particular, $X$ is an example of a quintic surface containing the twisted cubic $C$ for which $N_{C/X}$ is balanced, but $T_X|_C$ is not balanced.
\end{example}

\begin{remark}
    Since $\delta$ is the composition of $\psi$ with $\beta$, not every map in $\Hom(\cO(e+1)^e\oplus \cO(e)^{n-e}, \cO(de))$ can be realized as a $\delta$ for some hypersurface $X$. In fact, similarly to the map $\phi$, we can define the homomorphism
    \[ \phi_T: H^0(\cI_C(d))\to \Hom(\cO(e+1)^e\oplus \cO(e)^{n-e}, \cO(de)), \ F\mapsto \delta = \psi\circ \beta, \]
    which is given by taking the composition of $\phi$ with $\beta$. Both $\phi$ and $\phi_T$ share the same kernel $H^0(\cI_C^2(d))$. For $d\ge 3$, $\phi$ is surjective, and a dimension computation shows that the image of $\phi_T$ is a subspace of codimension $(ed-1)$ in $\Hom(\cO(e+1)^e\oplus \cO(e)^{n-e}, \cO(de))$.
\end{remark}

\subsection{Splitting of the tangent bundle for low-degree curves}

The tangent bundle sequence of $C$ in $X$ writes $T_X|_C$ as an extension of $N_{C/X}$ by $T_{\P^1}$:
\[ 0\longrightarrow \cO(2)\longrightarrow T_X|_C\longrightarrow N_{C/X}\longrightarrow 0. \]

We can tell when this sequence splits as a direct sum.

\begin{proposition}\label{split_if_normal_less_than_4}
    If $N_{C/X}\cong \bigoplus_{i=1}^{n-1}\cO(a_i)$ with $a_i < 4$ for all $i$, then $T_X|_C\cong N_{C/X}\oplus \cO(2)$.
\end{proposition}
\begin{proof}
    By Serre's duality for $\P^1$,
    \[ \Ext^1(N_{C/X}, \cO(2))\cong H^0((N_{C/X}\oplus\cO(-2))\otimes \cO(-2)) = H^0(N_{C/X}\otimes \cO(-4)) = 0 \]
    when $a_i<4$ for all $i$.
\end{proof}

When $X$ is a general hypersurface containing $C$, the normal bundle is balanced. When the degree of $X$ gets large enough with respect to the degree of $C$, the tangent bundle splits as $N_{C/X}\oplus \cO(2)$, so the tangent bundle $T_X|_C$ stops being balanced when $N_{C/X}$ has slope smaller than 1.

\begin{corollary}\label{slope_inequality_splits}
    Let $X$ be a degree $d$ general hypersurface in $\P^n$ containing the rational normal curve $C$ of degree $e$. If 
    \[ \mu (N_{C/X}) = \frac{e(n+1-d)-2}{n-2} \le 3, \]
    then $T_X|_C\cong N_{C/X}\oplus \cO(2)$, where $N_{C/X}$ is the balanced bundle of degree $ne+e-2$ and rank $n-2$. In particular, if $\mu (N_{C/X}) < 1$, then $T_X|_C$ is not balanced.
\end{corollary}
\begin{proof}
    If $X$ is general, then by \cite[Corollary 3.8 and Corollary 4.1]{CR}, the normal bundle $N_{C/X}$ is balanced, hence it is a sum of line bundles of degrees $\lfloor \mu(N_{C/X}) \rfloor$ and $\lceil \mu(N_{C/X}) \rceil$. Then the claim follows from Proposition \ref{split_if_normal_less_than_4}.
\end{proof}

We can explore the inequality in Corollary \ref{slope_inequality_splits} to highlight some cases when the restricted tangent bundle splits.

\begin{corollary}\label{split_in_degree_n}
    Let $X$ be a general hypersurface of degree $d$ in $\P^n$ containing the rational normal curve $C$ of degree $e\le n$. If $d\ge n = 3$ or $d+1\ge n\ge 4$, then $T_X|_C\cong N_{C/X}\oplus \cO(2)$, where $N_{C/X}$ is the balanced bundle of degree $e(n+1-d)-2$ and rank $n-2$. In particular, if $e<n=d$ or $e\le n\le d-1$, then $T_X|_C$ is not balanced for any hypersurface $X$.
\end{corollary}

We can also find cases when $T_X|_C$ cannot be balanced directly from the tangent bundle sequence. Notice that $T_X|_C$ is not balanced when $X$ is not Fano.
\begin{proposition}\label{tangent_degree_less_than_2}
    Let $X$ be a degree $d$ hypersurface containing a degree $e$ rational curve $C$. If
    \[ \mu(T_X|_C) = \frac{e(n+1-d)}{n-1} \le 1, \]
    then $T_X|_C$ is not balanced. Therefore, if $n < d$, or $n\ge d$ and $e\le \frac{n-1}{n+1-d}$, then $T_X|_C$ is not balanced.
\end{proposition}
\begin{proof}
    If $T_X|_C$ is balanced, then it is a sum of line bundles of degrees $\lfloor \mu(T_X|_C) \rfloor$ and $\lceil \mu(T_X|_C) \rceil$. So, if $\mu(T_X|_C)\le 1$, we could not have an injection $\cO(2)\to T_X|_C$, a contradiction.
\end{proof}

\subsection{Vector bundles on degenerations of rational curves}

Once we know the splitting type of the restricted tangent bundle for some curves on $X$, we can glue and smooth them to obtain higher-degree rational curves with a ``controlled'' restricted tangent bundle. If $C_1, C_2$ are rational curves on $X$ with $T_X|_{C_1}$ balanced and $T_X|_{C_2}$ perfectly balanced, we can get a curve of degree $\deg C_1+\deg C_2$ with a balanced restricted tangent bundle.

We summarize this in the following lemma on specialization of vector bundles on $\P^1$ to a gluing of two smooth rational curves. We refer to \cite{Smith_vector_bundles_trees_rational_curves} for a more complete discussion on the possible specializations of vector bundles on trees of rational curves.

\begin{lemma}\label{gluing_balanced_curves}(see \cite[Theorem 1.2]{Smith_vector_bundles_trees_rational_curves})
    Let $C = C_1\cup C_2$ be a nodal curve with $C_1, C_2\cong \P^1$ intersecting at one point $p$. Let $E$ be a rank $r$ vector bundle on $C$ such that 
    \[ E|_{C_1}\cong \bigoplus_{i=1}^r\cO(a_i) \  \text{ and } \ E|_{C_2}\cong \bigoplus_{i=1}^r\cO(b_i) \]
    with $\{a_i\}$ and $\{b_i\}$ in non-decreasing order. Assume that $E$ is the specialization of a vector bundle $E'$ on $\P^1$. Then the ``most unbalanced'' (in the sense of Lemma \ref{lemma_balanced_is_open_condition}) that $E'$ can be is
    \[ E'\cong \bigoplus_{i=1}^r\cO(a_i+b_i). \]

    In particular, if $E|_{C_1}$ is balanced, and $E|_{C_2}$ is perfectly balanced, then $E'$ is balanced.
\end{lemma}
\begin{proof}
    The obstructions in the splitting type of $E'$ come from the upper semicontinuity conditions:
    \[ h^0(C,E\otimes L)\ge h^0(\P^1,E'(\deg L)) \ \text{ and } \ h^1(C,E\otimes L)\ge h^1(\P^1,E'(\deg L)) \]
    for all line bundles $L$ on $C$. We have an exact sequence 
    \[ 0\longrightarrow E|_{C_1}(-p)\longrightarrow E\longrightarrow E|_{C_2}\longrightarrow 0. \]

    Denote by $\cO(a,b)$ the line bundle on $C$ that has degree $a$ on $C_1$ and degree $b$ on $C_2$. By twisting the exact sequence above by $L\cong \cO(-a_1,-b_1-1)$ and taking cohomologies, we get $h^1(C,E\otimes L) = 0$, hence $0\ge h^1(\P^1, E'(-a_1-b_1-1))$, thus $E'$ does not have summands of degree less than $a_1+b_1$. Similarly, if we take $L\cong \cO(-a_r,-b_r-1)$, we obtain $h^0(E\otimes L) = 0$, so $0\ge h^0(\P^1, E'(-a_r-b_r-1))$, thus $E'$ cannot have summands of degree larger than $a_r+b_r$. We repeat the argument for the other degrees to conclude the lemma.
\end{proof}

\subsection{Modular interpolation of rational curves}\label{Section_modular_interpolation}

An important property of a curve $C$ on a variety $X$ of dimension $n$ is its capacity to \textit{interpolate} a given number of general points in $X$ by deformations of $C$. We can make sense of this in terms of the space of curves on a variety going through $m$ general points of $X$, or in terms of morphisms $C\to X$ that send $m$ marked points in the curve $C$ to a fixed set of $m$ points in $X$. The first case is often called the \textit{interpolation property}, and is controlled by the normal bundle $N_{C/X}$. The latter is called \textit{modular interpolation}, and is connected to the positivity of the restricted tangent bundle $T_X|_C$. Both are studied for arbitrary genus curves in \cite{Interpolation_of_curves_on_Fano_hypersurfaces}. Since we are working with the tangent bundle here, we will use the word ``interpolation" as a synonym for modular interpolation.

In our case of rational curves and tangent bundles, we deal with maps $f: \P^1\to X$ and ask what is the maximum number $m$ of general points $x_1,\ldots , x_m$ of $X$ we can deform $f$ so that $f(p_i)=x_i$ for given $m$ general points $p_1, \ldots , p_m\in \P^1$. Let $f^*T_X\cong \bigoplus_{i=1}^n\cO(a_i)$ with $a_1\le \cdots \le a_n$ be the splitting of the restricted tangent bundle of such a map $f$. The space of morphisms $\P^1\to X$ with $p_i\mapsto x_i$ for $i=1,\ldots , m$ has tangent space at $[f]$ isomorphic to
\[ T_{[f]}\mathrm{Mor}(\P^1,X; p_i\mapsto x_i)\cong H^0(\P^1, f^*T_X\left (-p_1-\cdots -p_m)\right ) = H^0\left (\P^1, \bigoplus_{i=1}^{n}\cO(a_i-m)\right ), \]
and deformations of $f$ fixing $p_i\mapsto x_i$ will dominate $X$ if $a_1\ge m$ (see \cite[Corollary II.3.5.4]{Kollar_rational_curves_on_algebraic_varieties}). In this case, we can choose an additional point of $X$ for $f$ to interpolate. Hence, a curve $f$ with $f^*T_X\cong \bigoplus_{i=1}^n\cO(a_i)$ will interpolate up to $a_1+1$ general points in $X$. Equivalently, $f$ interpolates $m$ points while $H^1(f^*T_X(-m)) = 0$.

Notice that, among the vector bundles of $\P^1$ with fixed rank and degree, the balanced bundle has the largest $a_1$. In this sense, curves with a balanced restricted tangent bundle are the ones that interpolate the most points (see \cite[Corollary 20]{Interpolation_of_curves_on_Fano_hypersurfaces}). Observe that this maximum number of points is
\[ \lfloor \mu(f^*T_X)\rfloor + 1 = \left \lfloor \frac{-\deg f^*K_X}{n} \right \rfloor + 1. \]

When $X$ is a degree $d$ hypersurface in $\P^n$, and $f: \P^1\to X$ is a degree $e$ rational curve, the maximum number of points $f$ can interpolate is
\[ \left \lfloor \frac{e(n+1-d)}{n-1} \right \rfloor + 1, \]
which is achieved when $f^*T_X$ is balanced.

\begin{example}
    Degree $n$ rational normal curves $C$ in $\P^n$ have the nice properties \cite[Chapter 1]{Harris_First_Course}:
    \begin{itemize}
        \item Through $n+3$ points in linear general position in $\P^n$, there exists a unique rational normal curve;
        \item Given $n+2$ points $x_i$ in linear general position in $\P^n$, and $n+2$ distinct points $p_i\in \P^1$, there exists a unique rational normal curve $f:\P^1\to \P^n$ such that $f(p_i)=x_i$.
    \end{itemize}

    They correspond, respectively, to the splitting of the normal bundle $N_{C/\P^n}\cong \cO(n+2)^{n-1}$ and of the restricted tangent bundle $T_{\P^n}|_C\cong \cO(n+1)^n$.
\end{example}

\section{The inductive method}

The process to obtain a hypersurface $X$ with balanced restricted tangent bundle is done by first approaching the case $e=n$, and then doing induction on $n\ge e$.

For $e=n$, we have the tangent bundle sequence:
\[ 0\longrightarrow T_X|_C \overset{K_F}{\longrightarrow} \cO(n+1)^n \overset{\delta_F}{\longrightarrow}\cO(dn) \longrightarrow 0. \]

The strategy is to choose the appropriate polynomial $F$ so that its kernel $T_X|_C$ is balanced. To compute the kernel, we find independent \textit{column relations} satisfied by $\delta_F$. Then, create a matrix $K_{F}$ whose columns are the coefficients of the column relations. This construction implies that the map defined by $K_{F}$ factors through the kernel of $\delta_F$. Since they agree in rank and degree, it suffices to show that $K_{F}$ has maximum rank at every point $(s,t)$ of $\P^1$ to conclude that $K_{F}$ gives the kernel $T_X|_C$ of $\delta_F$.

\begin{example}\label{d3,e3,n3}
    Let $d=3$, $e=n=3$. The degree $3$ polynomial $F = x_0Q_{1,2} + x_3Q_{2,3}$ induces the map
    \[ \delta_F = \left(s^4t, -s^5+t^5, -st^4 \right ): \cO(4)^3\to \cO(9). \]\

    The columns $C_1, C_2, C_3$ of $\delta_F$ satisfy the relations:
    \begin{itemize}
        \item $s^2\cdot C_1 + st\cdot C_2 + t^2\cdot C_3 = 0;$
        \item $t^3\cdot C_1 + 0\cdot C_2 + s^3\cdot C_3 = 0$.
    \end{itemize}

    We use them as columns for the matrix
    \[ K_{F} = \begin{bmatrix}
            t^3 & s^2 \\
            0 & st\\
            s^3 & t^2 
            \end{bmatrix}. \]

    Notice that $K_{F}$ has maximum rank $2$ for all points $(s,t)\in \P^1$, hence the map
    $K_{F}: \cO(1)\oplus \cO(2)\to \cO(4)^3$ is the kernel of $\delta_F$, thus $T_X|_C\cong \cO(1)\oplus \cO(2)$.
\end{example}

Once the case $e=n$ is done, we approach the case $n>e$ by induction on $n$. First, observe that the rational normal curve $C$ spans a linear space $\Lambda\cong \P^e$, and that a general degree $d$ hypersurface $X\subset \P^n$ containing $C$ restricts to a general degree $d$ hypersurface $Y = X\cap \Lambda$ of $\P^e$. The inclusions define the following diagram of tangent bundle sequences:
\[\begin{tikzcd}
	& 0 & 0 \\
	0 & {T_Y|_C} & {\mathcal{O}(e+1)^e} & {\mathcal{O}(de)} & 0 \\
	0 & {T_X|_C} & {\mathcal{O}(e+1)^e\oplus \mathcal{O}(e)^{n-e}} & {\mathcal{O}(de)} & 0 \\
	& {\mathcal{O}(e)^{n-e}} & {N_{\mathbb{P}^e/\mathbb{P}^n}\cong \mathcal{O}(e)^{n-e}} \\
	& 0 & 0
	\arrow[from=1-2, to=2-2]
	\arrow[from=1-3, to=2-3]
	\arrow[from=2-1, to=2-2]
	\arrow["K_{f}", from=2-2, to=2-3]
	\arrow[from=2-2, to=3-2]
	\arrow["{\delta_f}", from=2-3, to=2-4]
	\arrow[from=2-3, to=3-3]
	\arrow[from=2-4, to=2-5]
	\arrow[equals, from=2-4, to=3-4]
	\arrow[from=3-1, to=3-2]
	\arrow["{K_{F}}", from=3-2, to=3-3]
	\arrow[from=3-2, to=4-2]
	\arrow["{\delta_F = (\delta_f; g)}", from=3-3, to=3-4]
	\arrow[from=3-3, to=4-3]
	\arrow[from=3-4, to=3-5]
	\arrow[equals, from=4-2, to=4-3]
	\arrow[from=4-2, to=5-2]
	\arrow[from=4-3, to=5-3]
\end{tikzcd}\]
where the map $\cO(e+1)^e\to \cO(e+1)^e\oplus \cO(e)^{n-e}$ is the identity on the first $e$ entries and zero elsewhere since $T_{\P^n}|_C\cong T_{\P^e}|_C\oplus N_{\P^e/\P^n}$.

If $X$ is defined by a polynomial $F = \sum_{i<j}F_{i,j}Q_{i,j} + \sum_{k=e+1}^{n}G_kx_k$ in $k[x_0,\ldots,x_n]$, then $Y$ is defined by the polynomial $f=\sum_{i<j}F_{i,j}|_{\Lambda}Q_{i,j}$ in $k[x_0,\ldots ,x_e]$. Thus, the map $\delta_F$ coincides with $\delta_f$ in its first $e$ entries; the last $(n-e)$ entries are the ones defined by the forms $G_k$. That is, we have $\delta_F = \left (\delta_f; g\right )$, where $g = \left ( G_{e+1}|_C,\cdots , G_n|_C\right )$. Our strategy is to work the diagram backwards, and use $f$ to inductively recover an $F$ so that $T_X|_C$ is balanced.

Suppose, by induction hypothesis, there exists a degree $d$ hypersurface $Y\subset \P^{n-1}$, for some $n > e$, defined by a polynomial $f\in k[x_0,\ldots , x_{n-1}]$ with $T_Y|_C$ balanced. It comes with its tangent bundle sequence
\[ 0\longrightarrow T_Y|_C \overset{K_{f}}{\longrightarrow} \cO(e+1)^e\oplus \cO(e)^{n-e-1} \overset{\delta_f}{\longrightarrow}\cO(de) \longrightarrow 0. \]

Now, say that $E$ is the balanced vector bundle of the same rank and degree as $T_X|_C$, that is, if $T_X|_C$ is balanced, then we should have $T_X|_C\cong E$. We then look for a pair of injective maps
\[ J: T_Y|_C\longrightarrow E \ \text{ and } \ N_1: E \longrightarrow \cO(e+1)^e\oplus \cO(e)^{n-e-1} \]
so that $K_{f} = N_1\cdot J$. By rank and degree considerations, the cokernel $N_2$ of $J$ maps $E$ to $\cO(e)$. These give us the map
\[ N = {\scriptscriptstyle\left(\begin{array}{c} N_1\\ \hline N_2\end{array}\right )}: E\to \cO(e+1)^e\oplus \cO(e)^{n-e-1} \]
that makes the following diagram commute:
\[\begin{tikzcd}
	& 0 & 0 \\
	0 & {T_Y|_C} & {\mathcal{O}(e+1)^e\oplus \mathcal{O}(e)^{n-e-1}} & {\mathcal{O}(de)} & 0 \\
	0 & E & {\mathcal{O}(e+1)^e\oplus \mathcal{O}(e)^{n-e}} & {\mathcal{O}(de)} & 0 \\
	& {\mathcal{O}(e)} & {\mathcal{O}(e)} \\
	& 0 & 0
	\arrow[from=1-2, to=2-2]
	\arrow[from=1-3, to=2-3]
	\arrow[from=2-1, to=2-2]
	\arrow["{K_{f}}", from=2-2, to=2-3]
	\arrow["J"', from=2-2, to=3-2]
	\arrow["{\delta_f}", from=2-3, to=2-4]
	\arrow[from=2-3, to=3-3]
	\arrow[from=2-4, to=2-5]
	\arrow[from=3-1, to=3-2]
	\arrow["N", from=3-2, to=3-3]
	\arrow["N_2"', from=3-2, to=4-2]
	\arrow["\delta", from=3-3, to=3-4]
	\arrow[from=3-3, to=4-3]
	\arrow[from=3-4, to=3-5]
	\arrow[equals, from=4-2, to=4-3]
	\arrow[from=4-2, to=5-2]
	\arrow[from=4-3, to=5-3]
\end{tikzcd}\]

Define $\delta$ the cokernel of $N$. By the commutativity of the diagram, up to a change of basis $\delta$ coincides with $\delta_f$ in its first $n-1$ entries, that is,
\[ \delta = \left(\delta_f; g\right ) \]
for some $g: \cO(e)\to \cO(de)$. By Lemma \ref{projectivelynormal}, there exists a degree $d-1$ polynomial $G_{n}$ so that 
\[ F = f + G_nx_n \]
defines $\delta_F = (\delta_f; g) = \delta$. Therefore, $N$ is the kernel of $\delta_F$, hence, $F$ defines $X$ with $T_X|_C\cong E$.

The following proposition shows that we can always obtain $X$ from $Y$.

\begin{proposition}\label{Ext_Proposition}
    Let $n > e$ and $Y\subset \P^{n-1}$ be a degree $d$ hypersurface containing the rational normal curve $C$ of degree $e$. Then, for any extension
    \[ 0\longrightarrow T_Y|_C\longrightarrow E\longrightarrow \cO(e)\longrightarrow 0 \]
    of $\cO(e)$ by $T_Y|_C$, there exists a degree $d$ hypersurface $X\subset \P^n$ such that $T_X|_C\cong E$. In particular, if $T_Y|_C$ is balanced for a general $Y$, then $T_X|_C$ is balanced for a general $X$.
\end{proposition}
\begin{proof}
    As above, $\delta_f$ is fixed for $Y$, and we look for a map $\delta = (\delta_f; g)$ whose kernel is $E$. Every $(\delta_f; g)$ comes from an $F = f + G_nx_n$ for some polynomial $G_n$. Therefore, it suffices to show that, for any extension $E$, there is a $g$ inducing $E$. In other words, it suffices to show the map
    \[ \Hom(\cO(e), \cO(de))\to \Ext^1(\cO(e), T_Y|_C), \ g\mapsto T_X|_C \]
    is surjective. Indeed, applying the functor $\Hom(\cO(e),-)$ to the short exact sequence
    \[ 0\longrightarrow T_Y|_C\longrightarrow \cO(e+1)^{e}\oplus \cO(e)^{n-e-1}\overset{\delta_f}{\longrightarrow} \cO(de)\longrightarrow 0, \]
    we obtain
    \[ \Hom(\cO(e),\cO(de))\longrightarrow \Ext^1(\cO(e),T_Y|_C)\longrightarrow \Ext^1(\cO(e), \cO(e+1)^{e}\oplus \cO(e)^{n-e-1}) = 0. \]
\end{proof}

Therefore, we only need to obtain $X$ for the case $n=e$, and then the case $n>e$ follows by induction. The following lemmas will help us find explicit matrices $J$ and $N_1$ for degrees $d=2,3,4$.

\begin{lemma}\label{injectivity of N_1N_2}
    Let $A,B,C,D$ be vector bundles over $\P^1$ with $\rk A \le \rk B$ and $\rk C \ge \rk D$, and maps $K: A\to B$, $J: A\to C$, $N_2: C\to D$, and $N_1: C\to B$ so that the diagram commutes:
    \[\begin{tikzcd}
	A & B \\
	C & {B\oplus D}
	\arrow["K", from=1-1, to=1-2]
	\arrow["J"', from=1-1, to=2-1]
	\arrow["\left(\begin{array}{c} \mathrm{Id}\\ \hline 0\end{array}\right )" {font=\tiny}, from=1-2, to=2-2]
	\arrow["\left(\begin{array}{c} N_1\\ \hline N_2\end{array}\right )" {font=\tiny}, from=2-1, to=2-2]
    \end{tikzcd}\]
    If, at all points in $\P^1$, $K$ and $N_2$ have maximum rank $\rk A$ and $\rk D$, respectively, then ${\scriptscriptstyle \left(\begin{array}{c} N_1\\ \hline N_2\end{array}\right )}$ has maximum rank $(\rk A + \rk D)$.
\end{lemma}
\begin{proof}
    Since $K = N_1\cdot J$, we have $\rk N_1\ge \rk K = \rk A$. And since $N_1$ and $N_2$ map to different summands, we have $\rk {\scriptscriptstyle \left(\begin{array}{c} N_1\\ \hline N_2\end{array}\right )} = \rk N_1 + \rk N_2\ge \rk A + \rk D$.
\end{proof}

The consequence of Lemma \ref{injectivity of N_1N_2} is that we only need to look for matrices $J$ and $N_1$ so that $K_{f} = N_1\cdot J$, and the injectivity of $N$ follows automatically. Next, we present the matrices $J$ that will be used in the induction, and show how to find the corresponding $N_1$. They will come in three kinds: $J_0$, $J_1$, and $J_2$, described in the following lemmas.

\begin{lemma}\label{lemma_J0}
    Let $A, B$ and $D$ be vector bundles in $\P^1$. Let $K: A\to B$ be any map. Then for
    \[ J_0 = {\scriptscriptstyle \left(\begin{array}{c} \mathrm{Id}\\ \hline 0\end{array}\right )}: A\to A\oplus D, \ \text{ and } \ N_1 = \left ( K \ \lvert \ 0\right ): A\oplus D\to B, \]
    we have $K = N_1\cdot J_0$.    
\end{lemma}
\begin{proof}
    It follows directly from the definitions.
\end{proof}

Notice that the cokernel of $J$, that we will use as $N_2$, is
\[ \coker J_0 = \left( 0 \ \lvert \ \mathrm{Id}\right ): A\oplus D\to D. \]

\begin{lemma}\label{lemma_J}
    Let $a, r, s$ be integers with $r\ge 1$, $s\ge 0$, and let $B\cong \bigoplus_i \cO(a_i)$ be a vector bundle over $\P^1$ with all $a_i > a$. Let $J_1: \cO(a)^r\oplus \cO(a+1)^s\to \cO(a)^{r-1}\oplus \cO(a+1)^{s+2}$ be the map defined by the matrix
    \[ J_1 = \begin{bmatrix}
        s & & & & \\
        0 & 1 &  &  &  \\
        \vdots &  &  & \ddots &  \\
        0 &  &  &  & 1 \\
        t & &  & & 0 
        \end{bmatrix}. \]

    Then for every map $K: \cO(a)^r\oplus \cO(a+1)^s\to B$ we can compute a map $N:\cO(a)^{r-1}\oplus \cO(a+1)^{s+2}\to B$ such that $K = N\cdot J_1$.
\end{lemma}

We remark that the cokernel of $J_1$, which will serve as our $N_2$, is
\[ \coker J_1 = \left (t, 0, \cdots , 0, -s\right ). \]

\begin{proof}
    Let $m = \rk B$. Write the matrix $N = \left (b_{i,j}\right )_{m\times (r+s+1)}$. Then
    \[ N\cdot J_1 = \begin{bmatrix}
        sb_{1,1}+tb_{1,r+s+1} & b_{1,2} & b_{1.3} & \cdots & b_{1,r+s} \\
        sb_{2,1}+tb_{2,r+s+1} & b_{2,2} & b_{2,3} & \cdots & b_{2,r+s} \\
        \vdots & \vdots & \vdots & \ddots & \vdots \\
        sb_{m,1}+tb_{m,r+s+1} & b_{m,2} & b_{m,3} & \cdots & b_{m,r+s} 
        \end{bmatrix}. \]

    Since $B$ has summands of degree larger than $a$ and $r\ge 1$, $K$ contains a column of degree at least one, which can be chosen as the first column of $N\cdot J_1$ above. Then, we can easily choose $b_{ij}$ so that $N\cdot J_1 = K$.
\end{proof}

\begin{lemma}\label{lemma_J2}
    Let $a, r, s$ be integers with $r\ge 2$, $s\ge 0$, and let $B\cong \bigoplus_i \cO(a_i)$ be a vector bundle over $\P^1$ with all $a_i > a$. Let $J_2: \cO(a)^r\oplus \cO(a+1)^s\to \cO(a)^{r-2}\oplus \cO(a+1)^{s+3}$ be the map defined by the matrix
    \[ J_2 = \begin{bmatrix}
        s & 0 &  &  &  \\
        0 & t &  &  &  \\
        0 & 0 & 1 &  &  \\
        \vdots & \vdots &  & \ddots &  \\
        0 & 0 &  &  & 1 \\
        t & s &  &  & 0 
        \end{bmatrix}. \]

    Then, for every map $K: \cO(a)^r\oplus \cO(a+1)^s\to B$ we find an $N_1: \cO(a)^{r-2}\oplus \cO(a+1)^{s+3}\to B$ such that $K = N_1\cdot J_2$.
\end{lemma}

Observe that the cokernel of $J_2$ is
\[ \coker J_2 = \left ( t^2, s^2, 0, \cdots , 0, -st \right ). \]

\begin{proof}
    Let $m = \rk B$. Let $N_1 = \left (b_{i,j}\right )_{m\times (r+s+1)}$. Then
    \[ N_1\cdot J_2 = \begin{bmatrix}
            sb_{1,1}+tb_{1,r+s+1} & tb_{1,2}+sb_{1,r+s+1} & b_{1,3} & b_{1,4} & \cdots  & b_{1,r+s} \\
            sb_{2,1}+tb_{2,r+s+1} & tb_{2,2}+sb_{2,r+s+1} & b_{2,3} & b_{2,4} & \cdots & b_{2,r+s} \\
            \vdots & \vdots & \vdots & \vdots & \ddots & \vdots \\
            sb_{m,1}+tb_{m,r+s+1} & tb_{m,2}+sb_{m,r+s+1} & b_{m,3} & b_{m,4} & \cdots & b_{m,r+s} 
            \end{bmatrix}. \]

    Write $K = (k_{i,j})_{m\times (r+s)}$. Since $B$ is a sum of terms of degree larger than $a$ and $r\ge 2$, at least the first two columns of $K$ have degree at least one. Let $d_{i,j} = \deg k_{i,j}$. Decompose the entries of the first two columns of $K$ as
    \[ k_{i,1} = sp_i + c_{i,1}t^{d_{i,1}} \text{ and } k_{i,2} = tq_i + c_{i,2}s^{d_{i,2}}, \]
    with $p_i,q_i\in k[s,t]$, $c_{i,1},c_{i,2}\in k$, for $1\le i\le m$.

    Then, we can choose
    \begin{align*}
      b_{i,1} & = p_i - c_{i,2}s^{d_{i,2}-2}t;\\
      b_{i,2} & = q_i - c_{i,1}st^{d_{i,1}-2};\\
      b_{i,r+s+1} & = c_{i,1}t^{d_{i,1}-1} + c_{i,2}s^{d_{i,2}-1};
    \end{align*}
    for $1\le i\le m$. And for all the other entries, we can pick $b_{i,j} = k_{i,j}$.
\end{proof}

\begin{example}\label{d3e3n>3}
    Let $d=3$, $e=3$, $n\ge 3$. In the example \ref{d3,e3,n3}, we let $n=3$ and showed that the polynomial $f = x_0Q_{1,2} + x_3Q_{2,3}$ induces $\delta_f = \left (s^4t, -s^5+t^5, -st^4\right )$ and a balanced kernel $T_Y|_C\cong \cO(1)\oplus \cO(2)$ given by
    \[ K_{f} = \begin{bmatrix}
            t^3 & s^2 \\
            0 & st \\
            s^3 & t^2 
            \end{bmatrix}. \]
            
    Now, for $n=4$, we want to find $F$ so that we get a kernel $T_X|_C\cong E =\cO(2)^3$. By Lemma \ref{lemma_J}, it suffices to choose
    \[ J = \begin{bmatrix}
            s & 0 \\
            0 & 1 \\
            t & 0 
            \end{bmatrix}, \]
    and there will exist a matrix $N_1$ such that $K_{f} = N_1\cdot J$. Indeed, we can follow the proof of the lemma to compute
    \[ N_1 = \begin{bmatrix}
            0 & s^2 & t^2 \\
            0 & st & 0 \\
            s^2 & t^2 & 0 
            \end{bmatrix}. \]

    Let $N_2 = \coker J = \left(t, 0, -s\right )$ and $N = {\scriptscriptstyle \left(\begin{array}{c} N_1\\ \hline N_2\end{array}\right )}$. We then obtain the commutative diagram
\[\begin{tikzcd}
	{\mathcal{O}(1)\oplus \mathcal{O}(2)} & {\mathcal{O}(4)^3} & {\mathcal{O}(9)} & {(n=e=3)} \\
	{\mathcal{O}(2)^3} & {\mathcal{O}(4)^3\oplus \mathcal{O}(3)} & {\mathcal{O}(9)} & {(n=4)} \\
	{\mathcal{O}(3)} & {\mathcal{O}(3)} && {}
	\arrow["{K_{f}}", from=1-1, to=1-2]
	\arrow["J"', from=1-1, to=2-1]
	\arrow["{\delta_f}", from=1-2, to=1-3]
	\arrow[from=1-2, to=2-2]
	\arrow["N", from=2-1, to=2-2]
	\arrow["{N_2}"', from=2-1, to=3-1]
	\arrow["\delta", from=2-2, to=2-3]
	\arrow[from=2-2, to=3-2]
	\arrow[equals, from=3-1, to=3-2]
\end{tikzcd}\]

    The map $N$ is injective by Lemma \ref{injectivity of N_1N_2}, which can also be directly checked. Similarly to the computation of the kernel, we can use the relations between the rows of $N$ to compute its cokernel
    \[ \delta = \coker N = \left (s^4t, -s^5+t^5, -st^4; s^3t^3 \right ). \]

    Notice that, as expected, we got $\delta = (\delta_f; g)$, with $g = s^3t^3$. Let $G_4 = x_0x_3$, so $G_4|_C = s^3t^3$. Then $F = f + G_4x_4 = (Q_{1,2}+Q_{2,3})+(x_0x_3)x_4$ induces $\delta_F = \delta$. Hence, $N$ is the kernel of $\delta_F$, and $T_X|_C\cong \cO(2)^3$.

    We can repeat the process for $n=5$. In this case, the balanced bundle $E$ is $\cO(2)^3\oplus \cO(3)$, which is just the $T_X|_C$ from the case $n=4$ plus a summand $\cO(3)$, so we can simply choose $J$ as $J_0$ in Lemma \ref{lemma_J0}, that is, 
    \[ J = \begin{bmatrix}
        1 &  &  \\
         & 1 &  \\
         &  & 1 \\
        0 & 0 & 0 
        \end{bmatrix}, \]
    so $N_2 = \left (0,0,0,1\right )$, and $N_1 = \left (K_{F} \ \lvert \ 0\right )$. We then simply obtain $\delta = (\delta_F; 0)$, and the new $F$ is the same as in the case $n=4$. Similarly for every $n\ge 4$, we get $T_X|_C\cong \cO(2)^3\oplus \cO(3)^{n-3}$ which are all induced by $F = (Q_{1,2}+Q_{2,3})+(x_0x_3)x_4$.
\end{example}

By Corollary \ref{split_in_degree_n}, we know how $T_X|_C$ decomposes when $n\le d+1$. Then, Proposition \ref{Ext_Proposition} allows us to use the induction on $n$ to settle the case $e\le d+1$.

\begin{theorem}\label{theorem_e_less_than_or_equal_to_d+1}
    Let $3\le e\le d+1$, $d\ge 3$ and $e\le n$. Let $X\subset \P^n$ be a general degree $d$ hypersurface containing the degree $e$ rational normal curve $C$.
    \begin{enumerate}
        \item If $n < d$, or $n\ge d$ and $e\le \frac{n-1}{n+1-d}$, then $T_X|_C$ is not balanced.
        \item If $n\ge d$ and $e>\frac{n-1}{n+1-d}$, then $T_X|_C$ is balanced.
    \end{enumerate}
\end{theorem}
\begin{proof}
    \begin{enumerate}
        \item This follows from Proposition \ref{tangent_degree_less_than_2}.

        \item Suppose first that $e=d$ or $e=d+1$. In both cases, we have $1\le \mu(N_{C/X})\le 3$, then by Corollary \ref{split_if_normal_less_than_4}, we have $T_X|_C\cong N_{C/X}\oplus \cO(2)$ with $T_X|_C$ balanced. Therefore, by Proposition \ref{Ext_Proposition} and induction on $n$, we can find $X$ with $T_X|_C$ balanced for all $n\ge e$.

        We can now assume $e\le d-1$. By Corollary \ref{split_in_degree_n}, there exists a degree $d$ hypersurface $Y\subset \P^e$ with $T_Y|_C\cong N_{C/Y}\oplus \cO(2)$, where $N_{C/Y}$ is the balanced bundle of degree $e(e+1-d)-2$ and rank $e-2$. That is, $T_Y|_C$ is written as a direct sum of line bundles of degrees $2$, $\lfloor \mu(N_{C/Y})\rfloor$ and $\lceil \mu(N_{C/Y})\rceil$, where
        \[ \mu(N_{C/Y}) = \frac{e(e+1-d)-2}{e-2}. \]

        Let $E$ be the balanced vector bundle of degree $n(e+1-d)$ and rank $n-1$, that is, if $T_X|_C$ is balanced, then we should have $T_X|_C\cong E$. By Proposition \ref{Ext_Proposition} and induction on $n$, it suffices to show that there is an injection $T_Y|_C\to E$. Notice that we have
        \[ \mu(N_{C/Y}) \le 0 \ \text{ since } \ e\le d-1, \]
        while
        \[ \mu(E) = \frac{e(n-d+1)}{n-1} > 1 \ \text{ for } \ e > \frac{n-1}{n+1-d}.  \]
        Therefore, $E$ has at least one summand of degree $\ge 2$, and all summands of degree larger than the summands of $N_{C/Y}$. Thus, we do have an injection $T_Y|_C\to E$, and it follows that there exists $X$ with $T_X|_C\cong E$.
    \end{enumerate}
\end{proof}

We remark that, since we know examples of hypersurfaces $Y$ with balanced normal bundle $N_{C/Y}$ from \cite{Mioranci_Normal_Bundle}, we can follow the proof of Theorem \ref{theorem_e_less_than_or_equal_to_d+1} and our induction method to construct explicit examples of hypersurfaces $X$ with balanced restricted tangent bundle as long as we can find the appropriate matrices $J$ and $N_1$ at each step.

We treat the cases $e=1$ and $e=2$ separately. We remark that, in both cases, the restricted tangent bundle splits as $N_{C/X}\oplus \cO(2)$.

\begin{theorem}\label{theorem_cases_e=1_and_e=2}
    Let $n\ge 3$ and $d\ge 3$. Let $X\subset \P^n$ be a general degree $d$ hypersurface containing the rational curve $C$.
    \begin{enumerate}
        \item If $C$ is a line, the restricted tangent bundle $T_X|_C$ is not balanced.
        \item If $C$ is a smooth conic, the restricted tangent bundle $T_X|_C$ is balanced if and only if $n\ge 2d-2$.
    \end{enumerate}
\end{theorem}
\begin{proof}
    \begin{enumerate}
        \item For $e=1$, Corollary \ref{slope_inequality_splits} holds with $\mu(N_{C/X})<1$.

        \item For $e=2$, Corollary \ref{slope_inequality_splits} holds with $\mu(N_{C/X})\le 3$, thus $T_X|_C\cong N_{C/X}\oplus \cO(2)$ with $N_{C/X}$ balanced. Since we have $\mu(N_{C/X})\ge 1$ if and only if $n\ge 2d-2$, it follows that $T_X|_C$ is balanced if and only if $n\ge 2d-2$.
    \end{enumerate}
\end{proof}

For $n\ge 4$, Ran \cite[Theorem 40]{Interpolation_of_curves_on_Fano_hypersurfaces} shows that a general degree $n$ Fano hypersurface in $\P^n$ contains a degree $e$ rational curve $C$ with balanced normal bundle for every $e\ge n-1$. We use this curve to produce hypersurfaces $X$ with balanced $T_X|_C$ for $d\le e\le 2d-2$.

\begin{theorem}\label{theorem_e_between_d_and_2d-2}
    Let $n\ge d\ge 3$, $n\ge 4$ and let $X\subset \P^n$ be a general degree $d$ hypersurface. Then $X$ contains a degree $e$ rational curve with balanced restricted tangent bundle for every $d \le e\le 2d-2$.
\end{theorem}
\begin{proof}
    We will work by induction on $n$ starting at $n=d$. So first, let $n=d<e\le 2d-2$. By \cite[Theorem 40]{Interpolation_of_curves_on_Fano_hypersurfaces}, there exists a degree $e$ rational curve $C\subset X$ with balanced normal bundle $N_{C/X}$. Notice that, for our degree range, $1 < \mu(N_{C/X})\le 3$. Then, by Corollary \ref{slope_inequality_splits}, $T_X|_C\cong N_{C/X}\oplus \cO(2)$ and $T_X|_C$ is balanced.
    
    Now, we repeat the proof of Proposition \ref{Ext_Proposition} to apply induction on $n$. Notice that, for $d < e\le 2d-2$, we have $e+1 < \mu(T_{\P^d}|_C)\le e+2$, and since $T_{\P^d}|_C$ is balanced for a general rational curve in $\P^d$ \cite[Theorem 2]{Ramella_la_stratification}, $T_{\P^d}|_C$ is a direct sum of terms $\cO(e+1)$ and $\cO(e+2)$. By induction hypothesis, suppose that for some $n\ge d$ and the curve $C$ above, there exists a degree $d$ hypersurface $Y\subset \P^{n}$ with balanced $T_Y|_C$. For the step $n+1$, we have the diagram:
    \[\begin{tikzcd}
	& 0 & 0 \\
	0 & {T_Y|_C} & {T_{\mathbb{P}^n}|_C} & {\mathcal{O}(de)} & 0 \\
	0 & E & {T_{\mathbb{P}^n}|_C\oplus \mathcal{O}(e)} & {\mathcal{O}(de)} & 0 \\
	& {\mathcal{O}(e)} & {\mathcal{O}(e)} \\
	& 0 & 0
	\arrow[from=1-2, to=2-2]
	\arrow[from=1-3, to=2-3]
	\arrow[from=2-1, to=2-2]
	\arrow[from=2-2, to=2-3]
	\arrow[from=2-2, to=3-2]
	\arrow["\delta", from=2-3, to=2-4]
	\arrow[from=2-3, to=3-3]
	\arrow[from=2-4, to=2-5]
	\arrow[equals, from=2-4, to=3-4]
	\arrow[from=3-1, to=3-2]
	\arrow[from=3-2, to=3-3]
	\arrow[from=3-2, to=4-2]
	\arrow["{(\delta;g)}", from=3-3, to=3-4]
	\arrow[from=3-3, to=4-3]
	\arrow[from=3-4, to=3-5]
	\arrow[equals, from=4-2, to=4-3]
	\arrow[from=4-2, to=5-2]
	\arrow[from=4-3, to=5-3]
    \end{tikzcd}\]

Notice that, since $C\subset \P^d$, we have $T_{\P^{n+1}}|_C\cong T_{\P^n}|_C\oplus \cO(e)$ and we get all maps $g\in \Hom(\cO(e),\cO(de))$ above. And again, applying the functor $\Hom(\cO(e),-)$ to the first row, we have
\[ \Hom(\cO(e),\cO(de))\longrightarrow \Ext^1(\cO(e),T_Y|_C)\longrightarrow \Ext^1(\cO(e), T_{\P^n}|_C) = 0, \]
where $\Ext^1(\cO(e), T_{\P^n}|_C) = 0$ since $T_{\P^n}|_C$ is a sum of terms $\cO(e),\cO(e+1)$ and $\cO(e+2)$. Thus, the map $\Hom(\cO(e),\cO(de))\to \Ext^1(\cO(e),T_Y|_C)$ is surjective, hence we get all extensions $E$ of $\cO(e)$ by $T_Y|_C$. In particular, there exists a degree $d$ hypersurface $X\subset \P^{n+1}$ with $T_X|_C\cong E$ balanced. Therefore, for all $n\ge d$, there exists a degree $d$ hypersurface $X\subset \P^n$ with $T_X|_C$ balanced.
\end{proof}

\begin{corollary}\label{Corollary_Section_3}
    Let $X$ be a general degree $d\ge 3$ Fano hypersurface in $\P^n$. Then $X$ contains rational curves of degree $e$ with balanced restricted tangent bundle for all $\frac{n-1}{n+1-d} < e\le 2d-2$.
\end{corollary}
\begin{proof}
    For Fano hypersurfaces, $n\ge d$. The corollary follows for $e=1$ and $2$ by Theorem \ref{theorem_cases_e=1_and_e=2}, for $3\le e\le d$ by Theorem \ref{theorem_e_less_than_or_equal_to_d+1} and for $d\le e\le 2d-2$ by Theorem \ref{theorem_e_between_d_and_2d-2}. Theorem \ref{theorem_e_between_d_and_2d-2} does not include the case $n=3$, which we prove in Theorem \ref{theorem_cubics}.
\end{proof}

\section{Quadrics}

In this section, we consider the case $d=2$, that is, when $X$ is a quadric hypersurface. As we will see in the next sections, the restricted tangent bundle becomes balanced for a curve of sufficiently large degree when $d\ge 3$. Quadrics are a special case, where odd-degree curves will never have a balanced restricted tangent bundle. In other words, we can interpolate fewer than expected points by deforming odd-degree curves on quadric hypersurfaces.

\begin{example}
    A smooth quadric surface $X\subset \P^3$ is isomorphic to $\P^1\times \P^1$, and a degree $e$ rational curve $C$ in $X$ corresponds to a bi-degree $(e_1,e_2)$ curve in $\P^1\times \P^1$, with $e_1+e_2=e$. Let $\pi_1,\pi_2:\P^1\times \P^1\to \P^1$ be the two natural projections. Then
    \[ T_{\P^1\times \P^1}|_C\cong (\pi_1^*T_{\P^1}\oplus \pi_2^*T_{\P^1})|_C\cong \cO_{\P^1}(2e_1)\oplus \cO_{\P^1}(2e_2). \]

    Hence, $T_{\P^1\times \P^1}|_C$ will be balanced exactly when $e_1=e_2$. In particular, it can be balanced for an even-degree curve but never for an odd-degree rational curve.
\end{example}

The quadric in $\P^5$ corresponds to another classical example of an unbalanced restricted tangent bundle, which is the case of most rational curves in Grassmannians.

\begin{example}
    A smooth quadric hypersurface in $\P^5$ is isomorphic to the Grassmannian $G(2,4)$ (see \cite[Chapter 6.2]{Griffiths-Harris}). The tangent bundle $T_{G(k,n)}$ of a Grassmannian splits as $T_{G(k,n)}\cong S^*\otimes Q$, where $S$ and $Q$ are the tautological and quotient bundle, respectively. As investigated in \cite[Lemma 33]{Mandal_Gr}, for $T_{G(k,n)}|_C$ to be balanced, both $S^*|_C$ and $Q|_C$ need to be balanced. But, for $G(2,4)$ and $C$ of odd degree $e=2m+1$, we will have $S^*|_C\cong Q|_C\cong \cO(m)\oplus \cO(m+1)$, hence $T_{G(k,n)}|_C\cong \cO(2m)\oplus \cO(2m+1)^2\oplus \cO(2m+2)$ is unbalanced. Notice that it can be balanced if $C$ has even degree.

    More generally, a general deformation of a degree $e$ rational curve in $G(k,n)$ will have a balanced restricted tangent bundle if and only if either $k|e$ or $(n-k)|e$ (see \cite[Example 21]{Interpolation_of_curves_on_Fano_hypersurfaces}).
\end{example}

Let $X$ be a degree $2$ hypersurface in $\P^n$ and $C\subset X$ a rational curve of degree $e$. From the tangent bundle sequence, we see that if $T_X|_C$ is balanced, then it is $\cO(e)^{n-1}$. Hence $C$ interpolates $e+1$ points exactly when $T_X|_C$ is balanced. We will show that an odd-degree curve cannot interpolate the expected number of points in a quadric hypersurface. For that, we will describe a method of constructing rational curves of a given degree via rational scrolls from \cite{Kollar_Quadratic_solutions}. Kollár studies degree $e$ maps $\P^1\to Q^n$ where $Q^n$ is a smooth quadric of dimension $n\ge 3$, and shows the following theorem.

\begin{theorem}\cite[Theorem 1]{Kollar_Quadratic_solutions}\label{Kollar_Theorem_1}
    Let $Q^n$ be a smooth quadric of dimension $n\ge 3$. Then
    \[ \mathrm{Mor}_e(\P^1, Q^n)\overset{bir}{\sim}  \left\{\begin{matrix}
    Q^n\times \P^{ne} \ \text{if } e \text{ is even, and} \\
    \OG(\P^1,Q^n)\times \P^{ne-n+3} \ \text{if } $e$ \text{ is odd,}
    \end{matrix}\right. \]
    where $\overset{bir}{\sim}$ denotes birational equivalence and $\OG(\P^1, Q^n)$ the orthogonal Grassmannian of lines in $Q^n$.
\end{theorem}

During the proof of Theorem \ref{Kollar_Theorem_1}, Kollár shows the following proposition, which relates curves of the same parity.

\begin{proposition}\cite[Proposition 26]{Kollar_Quadratic_solutions}\label{Kollar_proposition_26}
    Let $Q^n$ be a smooth quadric of dimension $n\ge 3$. Then
    \[ \mathrm{Mor}_e(\P^1, Q^n)\overset{bir}{\sim} \mathrm{Mor}_{e-2}(\P^1, Q^n)\times \P^{2n} \ \text{ for } e\ge 3. \]
\end{proposition}

Since degree $1$ maps are lines, we have the rational equivalence $\mathrm{Mor}_1(\P^1, Q^n)\overset{bir}{\sim} \OG(\P^1, Q^n)\times \P^3$. We will obtain the higher-degree curves by intersecting ruled surfaces with the quadric.

Let $C$ be a smooth projective curve and $\phi, \psi: C\to \P^{n}$ be two morphisms. We will consider the ruled surface swept out by the lines $\left <\phi(p), \psi(p)\right >$ for $p\in C$. If $\phi$ and $\psi$ coincide at a zero-dimensional subscheme $Z\subset C$, $\phi|_Z = \psi|_Z$, then we can construct a ruled surface $S(\phi,\psi)\subset \P^{n}$ from $\phi,\psi$ with $\deg S = \deg \phi + \deg \psi - \deg Z$ (see \cite[Section 2]{Kollar_Quadratic_solutions} for more details on the definition of $S$).

Let $X\subset \P^{n}$ be a smooth quadric. Suppose that $\phi:C\to X$ above maps to $X$ and $\psi:C\to \P^{n}$ is a morphism not contained in $X$. We get a ruled surface $S(\phi,\psi)$. The quadric and the ruled surface meet on the image of $\phi$ and on the residual intersection $R$. The degree of $R$ is $2\deg S - \deg \phi = \deg \phi + 2\deg \psi - 2\deg Z$.

The following proposition uses the construction in the proof of Proposition \ref{Kollar_proposition_26} to show that we can interpolate $m$ points in $X$ with rational curves of degree $e$ if and only if we can interpolate $m-2$ points with rational curves of degree $e-2$.

\begin{proposition}\label{interpolate_m_iff_interpolate_m-2}
    Let $n\ge 3$ and $m\le e+1$ be integers. Let $X\subset \P^n$ be a smooth quadric, and $p_1,\ldots, p_m$ be $m$ general points in $\P^1$. Then there exists a degree $e$ morphism $\phi_e:\P^1\to X$ with $\phi_e(p_i)=x_i$ for any general set of $m$ points $x_1,\ldots , x_m\in X$ if and only if there exists a degree $e-2$ morphism $\psi_{e-2}:\P^1\to X$ with $\psi(p_i)=y_i$ for any general set of $m-2$ points $y_1,\ldots , y_{m-2}\in X$.
\end{proposition}
\begin{proof}
    Let $H\subset \P^{n}$ be an auxiliary hyperplane, and fix the points $0,1,\infty\in \P^1$ without loss of generality.

    \begin{figure}[h]
        \centering
        \includegraphics[width=0.6\linewidth]{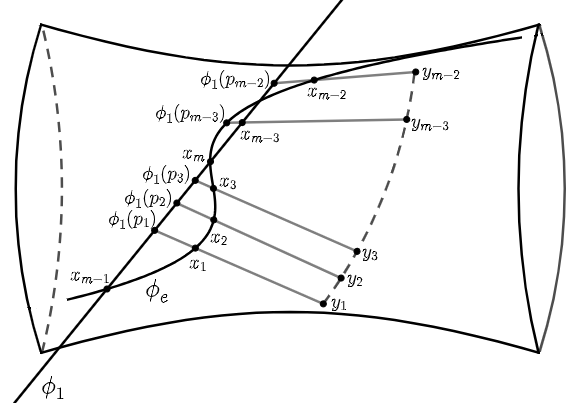}
    \end{figure}
    
    First, suppose that we can interpolate $m$ general points with curves of degree $e$, and let $y_1,\ldots , y_{m-2}$ be a set of $m-2$ general points in $X$. Choose $x_{m-1},x_m$ two general points in $X$. Let $\phi_1:\P^1\to \P^n$ be the line defined by $\phi_1(p_{m-1})=x_{m-1}$, $\phi_1(p_{m}) = x_m$ and $\phi(\infty) = \left <x_{m-1},x_m\right >\cap H$. This also sets the images $\phi_1(p_1), \ldots ,\phi_1(p_{m-2})\in \P^n$. For each $1\le i\le m-2$, the line $\left <\phi_1(p_i), y_i\right >$ meets $X$ at $y_i$ and at another point, which we name as $x_i$. Since $y_1,\ldots ,y_{m-2},x_{m-1},x_m$ were chosen as general points in $X$, then $x_1,\ldots , x_{m-2},x_{m-1},x_m$ are general points in $X$. Then, by hypothesis, there exists a degree $e$ morphism $\phi_e:\P^1\to X$ such that $\phi_1(p_i)=x_i$ for $i=1,\ldots , m$. Then, by \cite[Section 2]{Kollar_Quadratic_solutions}, there is a ruled surface $S(\phi_1,\phi_e)$ such that the residual of its intersection with $X$ is an irreducible curve $\psi_{e-2}:\P^1\to X$ of degree $e-2$ determined by the rulings of $S$. By construction, we have $\psi_{e-2}(p_i) = y_i$ for $1\le i\le m-2$.

    Conversely, suppose we can interpolate $m-2$ general points with rational curves of degree $e-2$. Let $x_1, \ldots , x_m$ be $m$ general points in $X$. Define $\phi_1:\P^1\to \P^n$ the line with $\phi_1(p_{m-1}) = x_{m-1}$, $\phi_1(p_m)=x_m$ and $\phi(\infty) = \left <x_{m-1},x_m\right >\cap H$. Then, for each $1\le i\le m-2$, the line $\left <\phi_1(p_i), x_i\right >$ meets $X$ at a second point $y_i$. By hypothesis, there exists a degree $e-2$ curve $\psi_{e-2}:\P^1\to X$ such that $\psi_{e-2}(p_i)=y_i$ for $i=1,\ldots , m-2$. Thus, $\phi_1$ and $\psi_{e-2}$ define a ruled surface $S(\phi_1,\psi_{e-2})$ whose intersection with $X$ is an irreducible curve $\phi_e:\P^1\to X$ of degree $e$ determined by the rulings of $S$, and such that $\phi_e(p_i) = x_i$ for $i=1,\ldots , m$.
\end{proof}

\begin{theorem}\label{theorem_odd_degree_curves_on_quadrics_are_not_balanced}
    Let $X\subset \P^n$ be a smooth quadric, and let $C\subset X$ be a rational curve of odd degree $e$. Then $T_X|_C$ is not balanced. Equivalently, deformations of $C$ do not interpolate $e+1$ general points of $X$.
\end{theorem}
\begin{proof}
    Since being balanced is an open condition, if $T_X|_C$ is balanced for some $X$, then it is for a general quadric. Thus, we may assume $X$ is general. Similarly, we can choose $C$ general in its family.

    If $e=1$, $C$ is a line, and we have $T_X|_C\cong \cO\oplus \cO(1)^{n-3}\oplus \cO(2)$, which is not balanced. In particular, lines interpolate up to $1$ point in $X$. Therefore, by Proposition \ref{interpolate_m_iff_interpolate_m-2} and induction on $e$, a curve of odd degree $e$ interpolates up to $e$ points. Hence, its restricted tangent bundle is not balanced.
\end{proof}

\begin{theorem}\label{theorem_quadrics}
    Let $X\subset \P^n$, $n\ge 3$, be a general quadric hypersurface containing a degree $e$, $1\le e\le n$, rational normal curve $C$.
    \begin{enumerate}
        \item If $e$ is even, then $T_X|_C\cong \cO(e)^{n-1}$.
        \item If $e$ is odd, then $T_X|_C\cong \cO(e-1)\oplus \cO(e)^{n-3}\oplus \cO(e+1)$.
    \end{enumerate}
    In addition, for each one of the cases above, we obtain an explicit example of a quadric $X$ with the corresponding $T_X|_C$ and balanced $N_{C/X}$.
\end{theorem}
\begin{proof}
The case $d=2$ is simpler, and we can show all the cases $n\ge e$ at the same time. Equivalently, we could show the case $n=e$ and run the induction with matrices $J$ as in Lemma \ref{lemma_J0} on every step.
    \begin{enumerate}[leftmargin=*]
        \item Suppose $e$ is even.  The tangent bundle sequence is
        \[ 0\longrightarrow T_X|_C\longrightarrow \cO(e+1)^e\oplus \cO(e)^{n-e}\overset{\delta}{\longrightarrow} \cO(2e)\longrightarrow 0. \]

        We choose the quadratic polynomial $F = Q_{1,2} + Q_{2,3} + \cdots + Q_{e-1,e}$, which induces the map $\psi_F: \cO(e+2)^{e-1}\oplus \cO(e)^{n-e}\to \cO(2e)$ given by the matrix
        \[ \psi_F = \left (s^{e-2}, s^{e-3}t, s^{e-4}t^2, \cdots , st^{e-3}, t^{e-2}; 0, 0, \cdots , 0\right ). \]
        
        And this defines the map $\delta_F = \psi_F\circ \beta : \cO(e+1)^e\oplus \cO(e)^{n-e}\to \cO(2e)$,
        \[ \delta_F = \left (ts^{e-2}, -s^{e-1}+s^{e-3}t^2, -s^{e-2}t+s^{e-4}t^3, \cdots , -s^2t^{e-3}+t^{e-1},-st^{e-2}; 0, \cdots ,0\right ). \]

        The columns $C_1, \ldots , C_n$ of $\delta_F$ satisfy the $n-1$ relations:
        \begin{itemize}[after=\vspace{\baselineskip}]
            \item $t\cdot C_i - s\cdot C_{i+1} = 0$ for $2\le i\le e-2$;
            \item $s\cdot C_1 + t\cdot C_2 + t\cdot C_4 + t\cdot C_6 + \cdots + t\cdot C_{e-2} + t\cdot C_e = 0$;
            \item $t\cdot C_1 + t\cdot C_3 + t\cdot C_5 + \cdots + t\cdot C_{e-5} + t\cdot C_{e-3} + s\cdot C_e = 0$;
            \item $1\cdot C_j = 0$ for $e+1\le j\le n$.
        \end{itemize}

        We remark that when $e=2$ the relations $s\cdot C_1 + t\cdot C_2 = 0$ and $1\cdot C_j = 0$, $3\le j\le n$, are satisfied.

        Hence, we define the matrix $K_{F}$ whose columns are the coefficients of the column relations of $\delta_F$:
        \[ K_{F} = \begin{bmatrix}
                0 & 0 & 0 &  & 0 & s & t &  &  &  \\
                t & 0 & 0 &  & 0 & t & 0 &  &  &  \\
                -s & t & 0 & \cdots & 0 & 0 & t &  &  &  \\
                0 & -s & t &  & 0 & t & 0 &  &  &  \\
                0 & 0 & -s &  & 0 & 0 & t &  &  &  \\
                \vdots & \vdots & \vdots & \ddots & \vdots & \vdots & \vdots &  &  &  \\
                0 & 0 & 0 &  & 0 & 0 & t &  &  &  \\
                0 & 0 & 0 &  & t & t & 0 &  &  &  \\
                0 & 0 & 0 & \cdots & -s & 0 & 0 &  &  &  \\
                0 & 0 & 0 &  & 0 & t & s &  &  &  \\
                 &  &  &  &  &  &  & 1 &  &  \\
                 &  &  &  &  &  &  &  & \ddots &  \\
                 &  &  &  &  &  &  &  &  & 1 
                \end{bmatrix}. \]
        
        It defines a map $\cO(e)^{n-1}\overset{K_{F}}{\longrightarrow} \cO(e+1)^e\oplus \cO(e)^{n-e}$ that factors through the kernel $T_X|_C$ of $\delta_F$. Thus, it suffices to show that  $K_{F}$ has maximum rank $n-1$ at every point $(s,t)\in \P^1$. This can be easily checked by dividing into the cases $s=1$ and $t=1$ and applying elementary row and column operations. Therefore, it follows that $K_{F}$ is the kernel of $\delta_F$ and $T_X|_C\cong \cO(e)^{n-1}$.

        \bigskip

        \item By Corollary \ref{slope_inequality_splits}, the claim follows for $e=1$. Assume that $e>1$. When $e$ is odd, we use the same polynomial $F = Q_{1,2}+Q_{2,3}+\cdots +Q_{e-1,e}$, which will induce the same map $\delta_F$. However, $\delta_F$ will not satisfy the column relations from the even case. Instead, we have:
        \begin{itemize}[after=\vspace{\baselineskip}]
            \item $C_1 + C_3 + \cdots + C_{e-3} + C_e = 0$;
            \item $-t\cdot C_i + s\cdot C_{e+1}$ for $2\le i\le e-2$;
            \item $(s^2-t^2)C_1 + (st)\cdot C_2 = 0$;
            \item $1\cdot C_j = 0$ for $e+1\le j\le n$.
        \end{itemize}

        Thus defining the matrix
        \[ K_{F} = \begin{bmatrix}
            1 & 0 & 0 & 0 &  & 0 & 0 & s^2-t^2 &  &  &  \\
            0 & -t & 0 & 0 &  & 0 & 0 & st &  &  &  \\
            1 & s & -t & 0 & \cdots & 0 & 0 & 0 &  &  &  \\
            0 & 0 & s & -t &  & 0 & 0 & 0 &  &  &  \\
            1 & 0 & 0 & s &  & 0 & 0 & 0 &  &  &  \\
            \vdots & \vdots & \vdots & \vdots & \ddots & \vdots & \vdots & \vdots &  &  &  \\
            0 & 0 & 0 & 0 &  & -t & 0 & 0 &  &  &  \\
            1 & 0 & 0 & 0 &  & s & -t & 0 &  &  &  \\
            0 & 0 & 0 & 0 & \cdots & 0 & s & 0 &  &  &  \\
            1 & 0 & 0 & 0 &  & 0 & 0 & 0 &  &  &  \\
             &  &  &  &  &  &  &  & 1 &  &  \\
             &  &  &  &  &  &  &  &  & \ddots &  \\
             &  &  &  &  &  &  &  &  &  & 1 
            \end{bmatrix}. \]
        
        Again, it is not difficult to check that $K_{F}$ has maximum rank at all points $(s,t)\in \P^1$, and thus defines an injection $\cO(e-1)\oplus \cO(e)^{n-3}\oplus \cO(e+1)\overset{K_{F}}{\longrightarrow} \cO(e+1)^e\oplus \cO(e)^{n-e}$. Hence $K_{F}$ gives the kernel of $\delta_F$, that is, $T_X|_C\cong \cO(e-1)\oplus \cO(e)^{n-3}\oplus \cO(e+1).$
    \end{enumerate}
    The proof of \cite[Theorem 4.3]{Mioranci_Normal_Bundle} with the polynomials $F$ above show they also induce a balanced normal bundle.
\end{proof}

\begin{theorem}\label{Corollary_quadrics}
    Let $X$ be a smooth quadric hypersurface in $\P^n$.
    \begin{enumerate}
        \item For every even $e\ge 2$, $X$ contains degree $e$ rational curves with balanced restricted tangent bundle $T_X|_C\cong \cO(e)^{n-1}$.

        \item For every odd $e\ge 1$, $X$ contains degree $e$ rational curves with restricted tangent bundle $T_X|_C\cong \cO(e-1)\oplus \cO(e)^{n-3}\oplus \cO(e+1)$.
    \end{enumerate}
\end{theorem}
\begin{proof}
    By Theorem \ref{theorem_quadrics}, $X$ contains lines $L$ with restricted tangent bundle $T_X|_L\cong \cO\oplus \cO(1)^{n-3}\oplus \cO(2)$ and conics $Q$ with perfectly balanced restricted tangent bundle $T_X|_Q\cong \cO(2)^{n-1}$. Then, by Lemma \ref{gluing_balanced_curves}, we can glue conics to $L$ and $Q$ to obtain curves $C$ of any degree $e$ and the desired restricted tangent bundle.
\end{proof}

\section{Cubics}

\begin{theorem}\label{theorem_cubics}
    Let $X\subset \P^n$ be a general cubic hypersurface containing a degree $e$ rational normal curve $C$.
    \begin{enumerate}
        \item If $e=1$, $C$ is a line, and we have the following cases:
         \[T_X|_C\cong  \left\{\begin{matrix*}[l]
                            \cO(-1)\oplus \cO(2), & \text{for } n=3; \\
                             \cO^2\oplus \cO(1)^{n-4}\oplus \cO(2), & \text{for } n\ge 4.
                            \end{matrix*}\right. \]
        \item If $e=2$, $C$ is a conic, and we have the following cases:
        \[T_X|_C\cong  \left\{\begin{matrix*}[l]
                            \cO\oplus \cO(2), & \text{for } n=3; \\
                             \cO(1)^2\oplus \cO(2)^{n-3}, & \text{for } n\ge 4.
                            \end{matrix*}\right. \]
        \item If $3\le e\le n$, we have:
        \[T_X|_C\cong  \left\{\begin{matrix*}[l]
                            \cO(e-2)\oplus \cO(e-1)^{e-2}, & \text{for } n=e;\\
                            \cO(e-1)^{e}\oplus \cO(e)^{n-e-1}, & \text{for } n>e.
                            \end{matrix*}\right. \]
    \end{enumerate}
    In addition, we obtain explicit examples of cubic hypersurfaces $X$ for each splitting type above.
\end{theorem}
\begin{proof}
    Cases (1) $e=1$ and (2) $e=2$ follow from Corollary \ref{slope_inequality_splits}. Examples with balanced normal bundle are found in \cite[Theorem 3.1]{Mioranci_Normal_Bundle}.
    \begin{enumerate}[leftmargin=*]
        \setcounter{enumi}{2}
        \item Let $e\ge 3$. First, suppose that the case $n=e$ holds, and let us show how to run the induction for $n>e$. Recall that, for each step, it suffices to find matrices $J$ and $N_1$ of correct dimension and degree such that $K_{F}=N_1\cdot J$. Thus, for $n=e+1$, it follows from Lemma \ref{lemma_J} with $J$ of the form $J_1$. And for $n\ge e+2$, the result follows from Lemma \ref{lemma_J0} with $J$ of the form $J_0$. The following diagram summarizes the process:
        \[\begin{tikzcd}
        	{\mathcal{O}(e-1)^{e-2}\oplus \mathcal{O}(e-2)} & {\mathcal{O}(e+1)^{e}} & {\mathcal{O}(de)} & {(n=e)} \\
        	{\mathcal{O}(e-1)^{e}} & {\mathcal{O}(e+1)^{e}\oplus \mathcal{O}(e)} & {\mathcal{O}(de)} & {(n=e+1)} \\
        	{\mathcal{O}(e-1)^{e}\oplus \mathcal{O}(e)} & {\mathcal{O}(e+1)^{e}\oplus \mathcal{O}(e)^2} & {\mathcal{O}(de)} & {(n=e+2)} \\
        	\vdots & \vdots & \vdots
        	\arrow["{K_{F_e}}", from=1-1, to=1-2]
        	\arrow["{J_1}"', from=1-1, to=2-1]
        	\arrow["{\delta_{F_e}}", from=1-2, to=1-3]
        	\arrow[from=1-2, to=2-2]
        	\arrow[equals, from=1-3, to=2-3]
        	\arrow["{K_{F_{e+1}}}", from=2-1, to=2-2]
        	\arrow["{J_0}"', from=2-1, to=3-1]
        	\arrow["{\delta_{F_{e+1}}}", from=2-2, to=2-3]
        	\arrow[from=2-2, to=3-2]
        	\arrow[equals, from=2-3, to=3-3]
        	\arrow["{K_{F_{e+2}}}", from=3-1, to=3-2]
        	\arrow["{J_0}"', from=3-1, to=4-1]
        	\arrow["{\delta_{F_{e+2}}}", from=3-2, to=3-3]
        	\arrow[from=3-2, to=4-2]
        	\arrow[equals, from=3-3, to=4-3]
        \end{tikzcd}\]
        
        Therefore, it suffices to show the case $n=e$. We work separately on the cases $n=3, n=4$, and $n\ge 5$. The case $n=3$ was done in the examples \ref{d3,e3,n3} and \ref{d3e3n>3}. For the case $n=4$, consider $F = x_0Q_{1,2} + x_1Q_{2,3} + x_2Q_{3,4}$. It defines the map
        \[ \delta_{F} = \left ( s^6t, -s^7+s^3t^4, -s^4t^3+t^7, -st^6 \right ): \cO(5)^4\longrightarrow \cO(12) \]
        satisfying column relations that induce the matrix 
        \[ K_{F} = \begin{bmatrix}
            t^2 & st & s^3 \\
            0 & t^2 & s^2t \\
            s^2 & 0 & st^2 \\
            st & s^2 & t^3 
            \end{bmatrix} \]
        which has maximum rank for all $(s,t)\in \P^1$. Hence, the restricted tangent bundle is balanced, $T_{X}|_C\cong \cO(2)\oplus \cO(3)^2$.

        Assume now $n\ge 5$. Let 
        \[ F = x_0Q_{1,2} + x_1Q_{2,3} + \cdots + x_{n-4}Q_{n-3,n-2} + x_{n-2}Q_{n-2,n-1} + x_nQ_{n-1,n}. \]
        
        It induces the map on normal bundles $\psi_F: \cO(n+2)^{n-1}\to \cO(3n)$,
        \[ \psi_F = \left ( s^{2n-2}, s^{2n-4}t^2, s^{2n-6}t^4, \cdots , s^6t^{2n-8}, s^3t^{2n-5}, t^{2n-2} \right ). \]

        Observe that the degree in $t$ increases by $2$ in each entry of $\psi_F$, except for the last two entries, when it increases by $3$. We then get $\delta_F: \cO(n+1)^n\to \cO(3n)$ of the form
        \begin{multline*}
            \delta = \psi_F\circ \beta = ( s^{2n-2}t, -s^{2n-1}+s^{2n-4}t^3, -s^{2n-3}t^2+s^{2n-6}t^5, -s^{2n-5}t^4+s^{2n-8}t^7, \cdots ,\\ -s^9t^{2n-10}+s^6t^{2n-7}, -s^7t^{2n-8}+s^3t^{2n-4}, -s^4t^{2n-5}+t^{2n-1}, -st^{2n-2} ).
        \end{multline*}

        It satisfies the following column relations:
        \begin{itemize}[after=\vspace{\baselineskip}]
            \item $-t^2\cdot C_i + s^2\cdot C_{i+1} = 0$ for $2\le i\le n-4$;
            \item $(s^3-t^3)\cdot C_1 + s^2t\cdot C_2 = 0$.
        \end{itemize}

        It satisfies three additional ``alternating relations" that depend on $n\mod 3$. The relations end with different coefficients at the last columns $C_n, C_{n-1}, C_{n-2}, C_{n-3}$, and then keep alternating the coefficients $t^2, st, 0, t^2,st, 0, \ldots$ 

        If $n\equiv 0\mod 3$, they are:
        \begin{itemize}[after=\vspace{\baselineskip}]
            \item $t^2\cdot C_n + st\cdot C_{n-1} + s^2\cdot C_{n-2} + 0\cdot C_{n-3} + st\cdot C_{n-4} + 0 \cdot C_{n-5} + t^2\cdot C_{n-6} + st\cdot C_{n-7} + 0\cdot C_{n-8} + \cdots + t^2\cdot C_3 + st\cdot C_2 + s^2\cdot C_1 = 0$;
            
            \item $st\cdot C_n + s^2\cdot C_{n-1} + 0\cdot C_{n-2} + t^2\cdot C_{n-3} + st\cdot C_{n-4} + 0\cdot C_{n-5} + t^2\cdot C_{n-6} + st\cdot C_{n-7} + \cdots + t^2\cdot C_3 + st\cdot C_2 + s^2\cdot C_1 = 0$;
            
            \item $s^2\cdot C_n + 0\cdot C_{n-1} + t^2\cdot C_{n-2} + st\cdot C_{n-3} + 0\cdot C_{n-4} + t^2\cdot C_{n-5} + st\cdot C_{n-6} + 0\cdot C_{n-7} + \cdots + st\cdot C_3 + 0\cdot C_2 + t^2\cdot C_1 = 0$.
        \end{itemize}

        For $n\equiv 1\mod 3$, the relations are the same, except they differ at the first coefficients due to the alternation. They are:
        \begin{itemize}[after=\vspace{\baselineskip}]
            \item $t^2\cdot C_n + st\cdot C_{n-1} + s^2\cdot C_{n-2} + 0\cdot C_{n-3} + st\cdot C_{n-4} + 0 \cdot C_{n-5} + t^2\cdot C_{n-6} + st\cdot C_{n-7} + 0\cdot C_{n-8} + \cdots + st\cdot C_3 + 0\cdot C_2 + t^2\cdot C_1 = 0$;
            \item $st\cdot C_n + s^2\cdot C_{n-1} + 0\cdot C_{n-2} + t^2\cdot C_{n-3} + st\cdot C_{n-4} + 0\cdot C_{n-5} + t^2\cdot C_{n-6} + st\cdot C_{n-7} + \cdots + st\cdot C_3 + 0\cdot C_2 + t^2\cdot C_1 = 0$;
            \item $s^2\cdot C_n + 0\cdot C_{n-1} + t^2\cdot C_{n-2} + st\cdot C_{n-3} + 0\cdot C_{n-4} + t^2\cdot C_{n-5} + st\cdot C_{n-6} + 0\cdot C_{n-7} + \cdots + 0\cdot C_3 + t^2\cdot C_2 + st\cdot C_1 = 0$.
        \end{itemize}

        And if $n\equiv 2\mod 3$, the relations are:
        \begin{itemize}[after=\vspace{\baselineskip}]
            \item $t^2\cdot C_n + st\cdot C_{n-1} + s^2\cdot C_{n-2} + 0\cdot C_{n-3} + st\cdot C_{n-4} + 0 \cdot C_{n-5} + t^2\cdot C_{n-6} + st\cdot C_{n-7} + 0\cdot C_{n-8} + \cdots + 0\cdot C_3 + t^2\cdot C_2 + st\cdot C_1 = 0$;
            \item $st\cdot C_n + s^2\cdot C_{n-1} + 0\cdot C_{n-2} + t^2\cdot C_{n-3} + st\cdot C_{n-4} + 0\cdot C_{n-5} + t^2\cdot C_{n-6} + st\cdot C_{n-7} + \cdots + 0\cdot C_3 + t^2\cdot C_2 + st\cdot C_1 = 0$;
            \item $s^2\cdot C_n + 0\cdot C_{n-1} + t^2\cdot C_{n-2} + st\cdot C_{n-3} + 0\cdot C_{n-4} + t^2\cdot C_{n-5} + st\cdot C_{n-6} + 0\cdot C_{n-7} + \cdots + t^2\cdot C_3 + st\cdot C_2 + s^2\cdot C_1 = 0$.
        \end{itemize}

        We exhibit here the matrix $K_F$ for $n\equiv 0\mod 3$:
        \[ K_F = \begin{bmatrix}
        0 &  &  &  & s^2 & s^2 & t^2 & s^3-t^3 \\
        -t^2 &  &  &  & st & st & 0 & s^2t \\
        s^2 & -t^2 &  &  & t^2 & t^2 & st &  \\
         & s^2 &  &  & 0 & 0 & t^2 &  \\
         &  &  &  & \vdots & \vdots & \vdots &  \\
         &  & \ddots &  & st & st & 0 &  \\
         &  &  &  & t^2 & t^2 & st &  \\
         &  &  &  & 0 & 0 & t^2 &  \\
         &  &  & -t^2 & st & st & 0 &  \\
         &  &  & s^2 & 0 & t^2 & st &  \\
         &  &  & 0 & s^2 & 0 & t^2 &  \\
         &  &  & 0 & st & s^2 & 0 &  \\
         &  &  & 0 & t^2 & st & s^2 &  
        \end{bmatrix}. \]

        We still need to show $K_F$ is injective to confirm it is the kernel of $\delta_F$. We claim it has maximum rank $n-1$ at all points $(s,t)$ in $\P^1$. Suppose $t=1$; it is similar for $s=1$. We can show it by Gauss-Jordan elimination. Send the first row to the last one, and use the $-t^2=-1$ along the diagonal as pivots to make their rows and columns into zeros. This reduces $K_F$ to
        \[ K_F\sim \begin{bmatrix}
        1 &  &  &  &  &  &  &  \\
         & 1 &  &  &  &  &  &  \\
         &  &  &  &  &  &  &  \\
         &  &  &  &  &  &  &  \\
         &  & \ddots &  &  &  &  &  \\
         &  &  &  &  &  &  &  \\
         &  &  &  &  &  &  &  \\
         &  &  & 1 &  &  &  &  \\
         &  &  &  & s^2P_1 & 1+s^2P_2 & s+s^2P_3 & s^{2(n-4)} \\
         &  &  &  & s^2 & 0 & 1 & 0 \\
         &  &  &  & s & s^2 & 0 & 0 \\
         &  &  &  & 1 & s & s^2 & 0 \\
         &  &  &  & s^2 & s^2 & 1 & s^3-1 
        \end{bmatrix}, \]
        where $P_1,P_2,P_3$ are polynomials in $s$. Thus, it suffices to show that 
        \[ \begin{bmatrix}
        s^2P_1 & 1+s^2P_2 & s+s^2P_3 & s^{2(n-4)} \\
        s^2 & 0 & 1 & 0 \\
        s & s^2 & 0 & 0 \\
        1 & s & s^2 & 0 \\
        s^2 & s^2 & 1 & s^3-1 
        \end{bmatrix} \]
        has rank $4$ for all $s$, which can be verified directly by computing its $4\times 4$ minors. Therefore, $K_F$ is the kernel of $\delta_F$, and we get $T_X|_C\cong \cO(n-1)^{n-2}\oplus \cO(n-2)$.
    \end{enumerate}
\end{proof}

\begin{corollary}
    Let $X\subset \P^n$ be a general cubic hypersurface. If $e=2$ and $n\ge 5$; or $e\ge 3$, then $X$ contains a rational curve of degree $e\le n$ with balanced restricted tangent bundle.
\end{corollary}

\section{Quartics}

\begin{theorem}\label{theorem_quartics}
    Let $X\subset \P^n$ be a general quartic hypersurface containing a degree $e$ rational normal curve $C$.
    \begin{enumerate}
        \item If $e=1$, we have the following cases:
        \[T_X|_C\cong  \left\{\begin{matrix*}[l]
                            \cO(-2)\oplus \cO(2), & \text{for } n=3; \\
                             \cO(-1)\oplus \cO\oplus \cO(2), & \text{for } n=4; \\
                             \cO^3\oplus \cO(1)^{n-5}\oplus \cO(2), & \text{for } n\ge 5.
                            \end{matrix*}\right. \]

    \item If $e=2$, we have:
    \[T_X|_C\cong  \left\{\begin{matrix*}[l]
                            \cO(-2)\oplus \cO(2), & \text{for } n=3; \\
                             \cO^2\oplus \cO(2), & \text{for } n=4; \\
                             \cO\oplus \cO(1)^{2}\oplus \cO(2), & \text{for } n=5; \\
                             \cO(1)^{4}\oplus \cO(2)^{n-5}, & \text{for } n\ge 6.
                            \end{matrix*}\right. \]

    \item If $e=3$, we have:
    \[T_X|_C\cong  \left\{\begin{matrix*}[l]
                            \cO(-2)\oplus \cO(2), & \text{for } n=3; \\
                             \cO\oplus \cO(1)\oplus \cO(2), & \text{for } n=4; \\
                             \cO(1)^{2}\oplus \cO(2)^2, & \text{for } n=5; \\
                             \cO(1)\oplus \cO(2)^{4}, & \text{for } n=6; \\
                             \cO(2)^{6}\oplus \cO(3)^{n-7}, & \text{for } n\ge 7.
                            \end{matrix*}\right. \]

    \item If $e\ge 4$, we have:
    \[T_X|_C\cong  \left\{\begin{matrix*}[l]
                            \cO(e-3)^{2}\oplus \cO(e-2)^{e-3}, & \text{for } n=e;\\
                            \cO(e-2)^{2e-n+1}\oplus \cO(e-1)^{2(n-e-1)}, & \text{for } 2e+1\ge n>e; \\
                            \cO(e-1)^{2e}\oplus \cO(e)^{n-2e-1}, & \text{for } n > 2e+1.
                            \end{matrix*}\right. \]
    \end{enumerate}
    In addition, we obtain explicit examples of quartic hypersurfaces $X$ for each splitting type above.
\end{theorem}

\begin{proof}
    The cases $e=1, 2, 3$ follow from Corollary \ref{slope_inequality_splits}. Examples with balanced normal bundle are shown in \cite[Theorem 3.1]{Mioranci_Normal_Bundle}.
    \begin{enumerate}[leftmargin=*]
    \setcounter{enumi}{3}
        \item Suppose first that we have proved the case $n=e$, and let us show how to apply the induction on $n$ to obtain the cases $n > e$. For each step, we have $K_F$ obtained from the previous step. Then, it suffices to find matrices $J$ and $N_1$ such that $K_F = N_1\cdot J$. For $n=e+1$, it follows from Lemma \ref{lemma_J2} with $J$ of the form $J_2$. For $2e+1\ge n > e$, it follows from Lemma \ref{lemma_J} with $J$ of the form $J_1$. And for $n>2e+1$, it follows from Lemma \ref{lemma_J0} with $J$ of the form $J_0$. 

        Therefore, it suffices to show the case $n=e$. We work the cases $n=4,5,6$ and $n\ge 7$ separately.
        
        Assume $n=4$. We choose $F = x_0^2Q_{1,2} + x_2^2Q_{2,3} + x_4^2Q_{3,4}$, which induces the map
        \[ \delta_F = \left (s^{10}t, -s^{11}+s^5t^6, -s^6t^5+t^{11}, -st^{10} \right ): \cO(5)^4\longrightarrow \cO(16). \]

        The column relations of $\delta_F$ define the kernel matrix
        \[ K_F = \begin{bmatrix}
            s^3 & t^4 & st^3 \\
            s^2t & 0 & t^4 \\
            st^2 & s^4 & 0 \\
            t^3 & s^3t & s^4 
            \end{bmatrix}. \]

        Since $K_F$ has maximum rank $3$, we have $T_X|_C\cong \cO(1)^2\oplus \cO(2)$. 
        
        Now, let $n=6$. We choose a slightly different polynomial in this case: $F = x_0^2Q_{1,2} + x_0x_3Q_{2,3} + x_3^2Q_{3,4} + x_3x_6Q_{4,5} + x_6^2Q_{5,6} + x_3^2Q_{3,6}$ induces $\delta_F : \cO(7)^6\to \cO(24)$ given by
        \[ \delta_F = \left ( s^{16}t, -s^{17}+s^{12}t^5, -s^{13}t^4+s^8t^9+s^6t^{11}, -s^9t^8+s^4t^{13}, -s^5t^{12}+t^{17}, -s^9t^8-st^{16} \right ). \]

        This map has kernel
        \[ K_F = \begin{bmatrix}
            t^3 & s^2t & s^3 & -s^2t^2-st^3 & s^2t^2-st^3 \\
            0 & st^2 & s^2t & -st^3-t^4 & st^3-t^4 \\
            s^3 & t^3 & st^2 & 0 & 0 \\
            s^2t & -s^2t & -s^3+t^3 & -s^4+s^2t^2-t^4 & -s^3t+st^3 \\
            st^2 & s^3 & 0 & -s^3t+st^3 & -s^2t^2 \\
            t^3 & s^2t & s^3 & -s^3t-s^2t^2+t^4 & s^4-st^3 
            \end{bmatrix}. \]

        Hence $T_X|_C\cong \cO(3)^2\oplus \cO(4)^3$. 

        Now, we work on the more general case $n\ge 7$. We consider polynomials $F$ of the form
        \begin{multline*}
             F = x_0^2Q_{1,2} + x_1^2Q_{2,3} + \cdots + x_{n-6}^2Q_{n-5,n-4}\\ + x_{n-5}x_{n-4}Q_{n-4,n-3} + x_{n-3}^2Q_{n-3,n-2} + x_{n-2,n-1}^2Q_{n-2,n-1} + x_n^2Q_{n-1,n}.
        \end{multline*}
        
        They induce a map $\psi_F: \cO(n+2)^{n-1}\to \cO(4n)$ on normal bundles,
        \[ \psi_F = \left ( s^{3n-2}, s^{3n-5}t^3, s^{3n-8}t^6, \cdots , s^{16}t^{3n-18}, s^{12}t^{3n-14}, s^8t^{3n-10}, s^4t^{3n-6}, t^{3n-2} \right ). \]

        The map $\psi_F$ starts with $s^{3n-2}$, and then the powers of $t$ increase by $3$ for each entry, except the last four entries, when it increases by $4$. It gives the map $\delta_F: \cO(n+1)^n\to \cO(4n)$:
        \begin{multline*}
            \delta_F = ( s^{3n-2}t, s^{3n-1}+s^{3n-5}t^4, -s^{3n-4}t^3+s^{3n-8}t^7, s^{3n-7}t^6+s^{3n-11}t^{10}, \cdots \\
            -s^{20}t^{3n-21}+s^{16}t^{3n-17}, -s^{17}t^{3n-18}+s^{12}t^{3n-13}, -s^{13}t^{3n-14}+s^8t^{3n-9},\\ -s^9t^{3n-10}+s^4t^{3n-5}, -s^5t^{3n-6}+t^{3n-1}, -st^{3n-2} ).
        \end{multline*}

        We look for the column relations of $\delta_F$ to define our kernel matrix $K_F$. We have $(n-6)$ ``simple relations":
        \begin{itemize}[after=\vspace{\baselineskip}]
            \item $-t^3\cdot C_i + s^3\cdot C_{i+1} = 0$ for $2\le i\le n-6$;
            \item $(s^4-t^4)\cdot C_1 + (s^3t)\cdot C_2 = 0$
        \end{itemize}

        Additionally, there are $5$ ``alternating relations" whose first coefficients depend on $n\mod 4$. We will display them for $n\equiv 0 \mod 4$. The other cases are very similar.
        
        The first $4$ relations end differently at $C_n, \ldots , C_{n-4}$ but then alternate the coeffiecients $t^3, st^2, s^2t, 0; t^3, st^2, s^2t, 0; ...$ then at $C_1$ they might break the sequence. They are:
        \begin{itemize}[after=\vspace{\baselineskip}]
            \item The one ending with $t^3$:\\
            $t^3\cdot C_n + st^2\cdot C_{n-1} + s^2t\cdot C_{n-2} + s^3\cdot C_{n-3} + 0\cdot C_{n-4} + t^3\cdot C_{n-5} + st^2\cdot C_{n-6} + s^2t\cdot C_{n-7} + 0\cdot C_{n-8} + t^3\cdot C_{n-9} + st^2\cdot C_{n-10} + s^2t\cdot C_{n-11} + 0\cdot C_{n-12} + \cdots + t^3\cdot C_3 + st^2\cdot C_2 + s^2t\cdot C_1 = 0;$
            \item The one ending with $st^2$:\\
            $st^2\cdot C_n + s^2t\cdot C_{n-1} + s^3\cdot C_{n-2} + 0\cdot C_{n-3} + t^3\cdot C_{n-4} + st^2\cdot C_{n-5} + s^2t\cdot C_{n-6} + 0\cdot C_{n-7} + t^3\cdot C_{n-8} + st^2\cdot C_{n-9} + s^2t\cdot C_{n-10} + 0\cdot C_{n-11} + t^3\cdot C_{n-12} + \cdots + st^2\cdot C_3 + s^2t\cdot C_2 + s^3\cdot C_1 = 0;$
            \item The one ending with $s^2t$:\\
            $s^2t\cdot C_n + s^3\cdot C_{n-1} + 0\cdot C_{n-2} + t^3\cdot C_{n-3} + st^2\cdot C_{n-4} + s^2t\cdot C_{n-5} + 0\cdot C_{n-6} + t^3\cdot C_{n-7} + st^2\cdot C_{n-8} + s^2t\cdot C_{n-9} + 0\cdot C_{n-10} + t^3\cdot C_{n-11} + st^2\cdot C_{n-12} + \cdots + s^2t\cdot C_3 + 0\cdot C_2 + t^3\cdot C_1 = 0;$
            \item The one ending with $s^3$:\\
            $s^3\cdot C_n + 0\cdot C_{n-1} + t^3\cdot C_{n-2} + st^2\cdot C_{n-3} + s^2t\cdot C_{n-4} + 0\cdot C_{n-5} + t^3\cdot C_{n-6} + st^2\cdot C_{n-7} + s^2t\cdot C_{n-8} + 0\cdot C_{n-9} + t^3\cdot C_{n-10} + st^2\cdot C_{n-11} + s^2t\cdot C_{n-12} + \cdots + 0\cdot C_3 + t^3\cdot C_2 + st^2\cdot C_1 = 0.$
        \end{itemize}

        The last relation ends at $C_{n-4}$ with coefficients $s^4, -t^4, s^2t^2-st^3, -s^2t^2$, and then repeats the sequence $t^4, st^3-t^4, s^2t^2-st^3, -s^2t^2$, except at the coefficient of $C_1$. It is:
        \begin{itemize}[after=\vspace{\baselineskip}]
            \item $s^4\cdot C_{n-4} - t^4\cdot C_{n-5} + (s^2t^2-st^3)\cdot C_{n-6} - s^2t^2\cdot C_{n-7} + t^4\cdot C_{n-8} + (st^3-t^4)\cdot C_{n-9} + (s^2t^2-st^3)\cdot C_{n-10} - s^2t^2\cdot C_{n-11} + t^4\cdot C_{n-12} + (st^3-t^4)\cdot C_{n-13} + (s^2t^2-st^3)\cdot C_{n-14} - s^2t^2\cdot C_{n-15} + \cdots + t^4\cdot C_4 + (st^3-t^4)\cdot C_3 + (s^2t^2-st^3)\cdot C_2 + (s^3t-s^2t^2)C_1 = 0$.
        \end{itemize}
        
        We display here the matrix $K_F$ when $n\equiv 0\mod 4$:
        \[ K_F = \begin{bmatrix}
            0 &  &  &  &  & s^2t & s^3 & t^3 & st^2 & s^3t-s^2t^2 & s^4-t^4 \\
            -t^3 &  &  &  &  & st^2 & s^2t & 0 & t^3 & s^2t^2-st^3 & s^3t \\
            s^3 & -t^3 &  &  &  & t^3 & st^2 & s^2t & 0 & st^3-t^4 &  \\
             & s^3 &  &  &  & 0 & t^3 & st^2 & s^2t & t^4 &  \\
             &  &  &  &  & s^2t & s^2t & 0 & t^3 & -s^2t^2 &  \\
             &  &  &  &  & \vdots & \vdots & \vdots & \vdots & \vdots &  \\
             &  &  &  &  & 0 & t^3 & st^2 & s^2t & t^4 &  \\
             &  & \ddots &  &  & s^2t & 0 & t^3 & st^2 & -s^2t^2 &  \\
             &  &  &  &  & st^2 & s^2t & 0 & t^3 & s^2t^2-st^3 &  \\
             &  &  &  &  & t^3 & st^2 & s^2t & 0 & st^3-t^4 &  \\
             &  &  &  &  & 0 & t^3 & st^2 & s^2t & t^4 &  \\
             &  &  & -t^3 &  & s^2t & 0 & t^3 & st^2 & -s^2t^2 &  \\
             &  &  & s^3 & -t^3 & st^2 & s^2t & 0 & t^3 & s^2t^2-st^3 &  \\
             &  &  &  & s^3 & t^3 & st^2 & s^2t & 0 & -t^4 &  \\
             &  &  &  & 0 & 0 & t^3 & st^2 & s^2t & s^4 &  \\
             &  &  &  & 0 & s^3 & 0 & t^3 & st^2 & 0 &  \\
             &  &  &  & 0 & s^2t & s^3 & 0 & t^3 & 0 &  \\
             &  &  &  & 0 & st^2 & s^2t & s^3 & 0 & 0 &  \\
             &  &  &  & 0 & t^3 & st^2 & s^2t & s^3 & 0 &  
            \end{bmatrix}. \]
        
        We can check $K_F$ is injective by showing it has rank $n-1$ at all points $(s,t)$ in $\P^1$. This can be done by Gauss-Jordan elimination. Consider $t=1$; the case $s=1$ is similar. Send the first row to the last position, and use the $-t^3=-1$ along the diagonal as pivots to make their rows and columns zero. This process reduces $K_F$ to
        \[ K_F\sim \begin{bmatrix}
            1 &  &  &  &  &  &  &  &  &  &  \\
             & 1 &  &  &  &  &  &  &  &  &  \\
             &  &  &  &  &  &  &  &  &  &  \\
             &  &  &  &  &  &  &  &  &  &  \\
             &  &  &  &  &  &  &  &  &  &  \\
             &  &  &  &  &  &  &  &  &  &  \\
             &  & \ddots &  &  &  &  &  &  &  &  \\
             &  &  &  &  &  &  &  &  &  &  \\
             &  &  &  &  &  &  &  &  &  &  \\
             &  &  &  &  &  &  &  &  &  &  \\
             &  &  & 1 &  &  &  &  &  &  &  \\
             &  &  &  & 1 &  &  &  &  &  &  \\
             &  &  &  & 0 & 1+s^3P_1 & s+s^3P_2 & s^2+s^3P_3 & s^3P_4 & -1+s^3P_5 & s^{3(n-6)} \\
             &  &  &  & 0 & 0 & 1 & s & s^2 & s^4 & 0 \\
             &  &  &  & 0 & s^3 & 0 & 1 & s & 0 & 0 \\
             &  &  &  & 0 & s^2 & s^3 & 0 & 1 & 0 & 0 \\
             &  &  &  & 0 & s & s^2 & s^3 & 0 & 0 & 0 \\
             &  &  &  & 0 & 1 & s & s^2 & s^3 & 0 & 0 \\
             &  &  &  & 0 & s^2 & s^3 & 1 & s & s^3-s^2 & s^4-1 
            \end{bmatrix}, \]
            where $P_1, \ldots , P_5$ are polynomials in $s$. Thus, it suffices to show that
            \[ \begin{bmatrix}
            1+s^3P_1 & s+s^3P_2 & s^2+s^3P_3 & s^3P_4 & -1+s^3P_5 & s^{3(n-6)} \\
            0 & 1 & s & s^2 & s^4 & 0 \\
            s^3 & 0 & 1 & s & 0 & 0 \\
            s^2 & s^3 & 0 & 1 & 0 & 0 \\
            s & s^2 & s^3 & 0 & 0 & 0 \\
            1 & s & s^2 & s^3 & 0 & 0 \\
            s^2 & s^3 & 1 & s & s^3-s^2 & s^4-1 
            \end{bmatrix} \]
            has rank $6$ for all $s$, which can be done directly by computing its $6\times 6$ minors. Therefore, $T_X|_C\cong \cO(n-3)^2\oplus \cO(n-2)^{n-3}$.
    \end{enumerate}
\end{proof}

\begin{corollary}
    Let $X\subset \P^n$ be a general quartic hypersurface. If $e=2$ and $n\ge 6$; or $e=3$ and $n\ge 5$; or $e\ge 4$, then $X$ contains a degree $e\le n$ rational curve with balanced restricted tangent bundle.
\end{corollary}

\section{Higher-degree curves}

\begin{theorem}\label{theorem_higher_degree}
    Let $X\subset \P^n$ be a general degree $d\ge 4$ hypersurface containing a degree $e\le n$ rational normal curve $C$. If $e \ge 2d-2$, then the restricted tangent bundle $T_X|_C$ is balanced.
\end{theorem}

\begin{proof}
    By Proposition \ref{Ext_Proposition} and induction on $n$, it suffices to prove the theorem for $e=n$.

    By \cite[Corollary 3.8]{CR}, the normal bundle $N_{C/X}$ is balanced, and for $n\ge 2d-2$ it has the form
    \[ N_{C/X}\cong \cO(n+2-d)^{2d-4}\oplus \cO(n+3-d)^{n-2d+2}. \]

    It is induced by a map $\psi_F: \cO(n+2)^{n-1}\to \cO(dn)$ having $2d-4$ column relations with degree $d$ and $n-2d+2$ columns relations with degree $d-1$. To obtain such a $\psi_F$, we start with the entry $s^{dn-n-2}$ and increase the powers of $t$ by $d-1$ for the first $n-2d+2$ entries, and then increase it by $d$ for the remaining ones. That is, we use the following $\psi_F$:
    \begin{multline*}
        \psi_F = ( s^{dn-n-2}, s^{(dn-n-2)-(d-1)}t^{d-1}, s^{(dn-n-2)-2(d-1)}t^{2(d-1)}, \cdots,\\
        s^{(dn-n-2)-(n-2d+1)(d-1)}t^{(n-2d+1)(d-1)}, s^{(2d-4)d}t^{(dn-n-2)-(2d-4)d}, s^{(2d-3)d}t^{(dn-n-2)-(2d-3)d} , \cdots , \\
        s^{2d}t^{(dn-n-2)-2d}, s^{d}t^{(dn-n-2)-d}, t^{dn-n-2} ).
    \end{multline*}

    We know this $\psi_F$ is indeed induced by a degree $d$ polynomial $F$ by Proposition \ref{surjection_of_phi}. It is not difficult to obtain examples of $F$ for a given $\psi_F$. Hence, this same polynomial induces the map on tangent bundles $\delta: \cO(n+1)^n\to \cO(dn)$:
    \begin{multline*}
        \delta_F = \psi_F\circ \beta = (s^{dn-n-2}t, -s^{dn-n-1}+s^{(dn-n-2)-(d-1)}t^{d},\\
        -s^{(dn-n-2)-(d-1)+1}t^{d-1}+s^{(dn-n-2)-2(d-1)}t^{2(d-1)+1}, \cdots ,\\
        -s^{(dn-n-2)-(n-2d+1)(d-1)+1}t^{(n-2d+1)(d-1)}+s^{(2d-4)d}t^{(dn-n-2)-(2d-4)d+1},\\
        -s^{(2d-4)d+1}t^{(dn-n-2)-(2d-4)d}+s^{(2d-3)d}t^{(dn-n-2)-(2d-3)d+1},\cdots ,\\
        -s^{2d+1}t^{(dn-n-2)-2d}+s^dt^{(dn-n-2)-d+1},-s^{d+1}t^{(dn-n-2)-d}+t^{dn-n-1},-st^{dn-n-2}).
    \end{multline*}

    Call the $n$ entries of $\delta_F$ by $C_1, \ldots , C_n$. We will compute the kernel of $\delta_F$ by finding $n-1$ independent relations between these entries. There are $n-2d+1$ relations of degree $d-1$ of the form
    \begin{itemize}
        \item $-t^{d-1}C_i + s^{d-1}C_{i+1} = 0$ for $2\le i\le n-2d+2$;
    \end{itemize}
    and $d-4$ relations of degree $d$ given by
    \begin{itemize}
        \item $-t^dC_i + s^dC_{i+1} = 0$ for $n-2d+4\le i\le n-d-1$.
    \end{itemize}

    We also have $d$ ``alternating relations" of degree $d-1$. The first four end with 
    \begin{itemize}
        \item $s^it^{d-1-i}C_n + s^{i+1}t^{d-2-i}C_{n-1}+\cdots + s^{d-2}tC_{n-(d-2-i)} + s^{d-1}C_{n-(d-1-i)}+\cdots $
    \end{itemize}
    for $0\le i\le 3$, and repeat the sequence of coefficients $0, t^{d-1}, st^{d-2}, \ldots , s^{d-2}t$ for the remaining entries. We repeat this sequence as it is until $C_2$, whose coefficient will depend on $n \mod d$. The coefficient of $C_1$ might differ from the sequence: if the next term in the sequence is $0$, then use $s^{d-1}$ instead; otherwise, use the expected coefficient. For example, if $n\equiv 0 \mod d$, then the relation ending with $t^{d-1}$ is:
    \begin{multline*}
        (t^{d-1}C_n + st^{d-2}C_{n-1} + \cdots + s^{d-1}C_{n-d+1}) \\
        + (0\cdot C_{n-d} + t^{d-1}C_{n-d-1} + st^{d-2}C_{n-d-2} + \cdots + s^{d-2}tC_{n-2d+1}) \\
        + (0\cdot C_{n-2d} + t^{d-1}C_{n-2d-1}+\cdots +s^{d-2}tC_{n-3d+1}) + \cdots\\
        + (0\cdot C_{d} + t^{d-1}C_{d-1}+\cdots +s^{d-2}tC_{1}) = 0. 
    \end{multline*}

    The next relation, ending with $s^4t^{d-5}C_n$, ends with
    \begin{itemize}
        \item $(s^4t^{d-5}C_n + s^{5}t^{d-6}C_{n-1}+\cdots + s^{d-2}tC_{n-(d-6)} + s^{d-1}C_{n-(d-5)})+ 0\cdot C_{n-d+4} +(t^{d-1}C_{n-d+3} + st^{d-2}C_{n-d+2}+\cdots + s^{d-1}C_{n-2d+4}) + 0\cdot C_{n-2d+3} + (st^{d-2}C_{n-2d+2}+\cdots $
    \end{itemize}
    and then they start repeating the sequence $0,t^{d-1},\ldots , s^{d-2}t$ as for the four ones above. Notice it skips the coefficient $t^{d-1}$ that would be in $C_{n-2d+2}$.

    The remaining $d-5$ alternating relations end with $s^it^{d-1-i}C_n$ for $5\le i\le d-1$. They are similar to the relation above, but they end with
    \begin{itemize}
        \item $(s^it^{d-1-i}C_n + s^{i+1}t^{d-2-i}C_{n-1}+\cdots + s^{d-2}tC_{n-(d-2-i)} + s^{d-1}C_{n-(d-1-i)})+ 0\cdot C_{n-d+i} +(t^{d-1}C_{n-d+i-1} + st^{d-2}C_{n-d+i-2}+\cdots + s^{d-1}C_{n-2d+i}) +\cdots $
    \end{itemize}
    for $5\le i\le d-1$ and then they start repeating the sequence $0,t^{d-1},\ldots , s^{d-2}t$ as for the five ones above. The reason we divide them into these three groups is due to the $(2d-4)$ column relations of degree $d$ followed by the $n-2d+2$ relations of degree $d-1$ of $\psi_F$, which divide $\psi_F$ into two parts.

    We also have the degree $d$ relation
    \begin{itemize}
        \item $(s^d-t^d)C_1 + s^{d-1}tC_2 = 0$.
    \end{itemize}

    And finally, an additional alternating relation of degree $d$. It ends at $C_{n-2d+4}$ with the sequence of coefficients:
    \begin{itemize}
        \item $s^dC_{n-2d+4} - t^dC_{n-2d+3} + (s^2t^{d-2}-st^{d-1})C_{n-2d+2} + (s^3t^{d-3}-s^2t^{d-2})C_{n-2d+1} + \cdots + (s^{d-2}t^2-s^{d-3}t^3)C_{n-3d+6} + (-s^{d-2}t^2)C_{n-3d+5} + \cdots$
    \end{itemize}
    then, for the remaining entries, we repeat the sequence of coefficients $t^d, st^{d-1}-t^{d}, s^2t^{d-2}-st^{d-1},\ldots , s^{d-2}t^{2}-s^{d-3}t^3, -s^{d-2}t^2$. As with the other alternating relations, the sequence has $d$ terms, then the coefficient of $C_1$ will depend on $n \mod d$.
    
    These give us all the relations we need. They form the columns of the matrix $K_F$, which we show here for the case $n\equiv 0 \mod d$:
    
    \[ \begin{bsmallmatrix}
0 &  &   & s^{d-2}t & s^{d-1} & t^{d-1} & st^{d-2} & s^2t^{d-3} & s^3t^{d-4} & s^4t^{d-5} &  & s^{d-5}t^4 & s^{d-4}t^3 & s^d-t^d &  s^3t^{d-3}-s^2t^{d-2}  &  &  &  \\
-t^{d-1} &  &  & s^{d-3}t^2 & s^{d-2}t & 0 & t^{d-1} & st^{d-2} & st^{d-2} & s^3t^{d-4} &  & s^{d-6}t^5 & s^{d-5}t^4 & s^{d-1}t &  s^2t^{d-2}-st^{d-1} &  &  &  \\
s^{d-1} &  &  & \vdots & s^{d-3}t^2 & s^{d-2}t & 0 & t^{d-1} & t^{d-1} & st^{d-2} & \cdots & s^{d-7}t^6 & s^{d-6}t^5 & 0 & st^{d-1}-t^d &  &  &  \\
 &  &  & \vdots & \vdots & s^{d-3}t^2 & s^{d-2}t & 0 & 0 & t^{d-1} &  & s^{d-8}t^7 & s^{d-7}t^6 & 0 & t^{d} &  &  &  \\
 &  &  & t^{d-1} & \vdots & \vdots & s^{d-3}t^2 & s^{d-2}t & s^{d-2}t & 0 &  & s^{d-9}t^8 & s^{d-8}t^7 & 0 & \vdots &  &  &  \\
  &  &   & 0 & t^{d-1} & \vdots & \vdots & s^{d-3}t^2 & s^{d-3}t^2 & s^{d-2}t & \ddots & s^{d-10}t^9 & s^{d-9}t^8 & 0 & \vdots &  &  &  \\
 &  &  & s^{d-2}t & 0 & t^{d-1} & \vdots & \vdots & \vdots & s^{d-3}t^2 &  & s^{d-11}t^{10} & s^{d-10}t^9 & \vdots & -s^{d-2}t^2 &  &  &  \\
 & \ddots &  & \vdots & s^{d-2}t & 0 & t^{d-1} & \vdots & \vdots & \vdots &  & \vdots & \vdots & \vdots & \vdots &  &  &  \\
 &   &   & st^{d-2} & \vdots & s^{d-2}t & 0 & t^{d-1} & t^{d-1} & \vdots & \ddots & \vdots & \vdots &  & s^{d-1}t-t^d &  &  &  \\
 &  &   & t^{d-1} & st^{d-2} & \vdots & s^{d-2}t & 0 & 0 & t^{d-1} &   & s^8t^{d-9} & s^9t^{d-10} &  & t^d &  &  &  \\
 &  &  & 0 & t^{d-1} & st^{d-2} & \vdots & s^{d-2}t & s^{d-2}t & 0 &  & s^7t^{d-8} & s^8t^{d-9} &  & -s^{d-2}t^2 &  &  &  \\
 &   &  & s^{d-2}t & 0 & t^{d-1} & st^{d-2} & \vdots & \vdots & s^{d-2}t & \ddots & s^6t^{d-7} & s^7t^{d-8} &  & \vdots &  &  &  \\
 &  & -t^{d-1} & s^{d-3}t^2 & s^{d-2}t & 0 & t^{d-1} & st^{d-2} & st^{d-2} & \vdots &  & \vdots & s^6t^{d-7} &  & s^2t^{d-2}-st^{d-1} &  &  &  \\
 &  & s^{d-1} & s^{d-4}t^3 & s^{d-3}t^2 & s^{d-2}t & 0 & 0 & t^{d-1} & st^{d-2} &  & s^{d-7}t^6 & \vdots &  & -t^d &  &  &  \\
 &  &  & s^{d-5}t^4 & s^{d-4}t^3 & s^{d-3}t^2 & s^{d-2}t & s^{d-1} & 0 & t^{d-1} &  & s^{d-8}t^7 & s^{d-7}t^6 &  & s^d & -t^d &  &  \\
 &  &  & s^{d-6}t^5 & s^{d-5}t^4 & s^{d-4}t^3 & s^{d-3}t^2 & s^{d-2}t & s^{d-1} & 0 &  & s^{d-9}t^8 & s^{d-8}t^7 &  & 0 & s^d &   &  \\
 &  &  & \vdots & \vdots & \vdots & \vdots & \vdots & \vdots & \vdots & \ddots & \vdots & \vdots &  &  &  & \ddots &  \\
 &  &  & t^{d-1} & st^{d-2} & s^2t^{d-3} & s^3t^{d-4} & s^4t^{d-5} & s^5t^{d-6} & s^6t^{d-7} &  & s^{d-2}t & s^{d-1} &  & 0 &  &  & -t^d \\
 &  &  & 0 & t^{d-1} & st^{d-2} & s^2t^{d-3} & s^3t^{d-4} & s^4t^{d-5} & s^5t^{d-6} &  & s^{d-3}t^2 & s^{d-2}t &  & 0 &  &  & s^d \\
 &  &  & s^{d-1} & 0 & t^{d-1} & st^{d-2} & s^2t^{d-3} & s^3t^{d-4} & s^4t^{d-5} &  & s^{d-4}t^3 & s^{d-3}t^2 &  & 0 &  &  & 0 \\
 &  &  & \vdots & \vdots & \vdots & \vdots & \vdots & \vdots & \vdots & \cdots & \vdots & \vdots &  & \vdots &  &  & \vdots \\
 &  &  & s^2t^{d-3} & s^3t^{d-4} & s^4t^{d-5} & s^5t^{d-6} & s^6t^{d-7} & s^7t^{d-8} & s^8t^{d-9} &  & 0 & t^{d-1} &  & 0 &  &  & 0 \\
 &  &  & st^{d-2} & s^2t^{d-3} & s^3t^{d-4} & s^4t^{d-5} & s^5t^{d-6} & s^6t^{d-7} & s^7t^{d-8} &  & s^{d-1} & 0 &  & 0 &  &  & 0 \\
 &  &  & t^{d-1} & st^{d-2} & s^2t^{d-3} & s^3t^{d-4} & s^4t^{d-5} & s^5t^{d-6} & s^6t^{d-7} & \cdots & s^{d-2}t & s^{d-1} &  & 0 &  &  & 0
\end{bsmallmatrix}. \]

Now, we are left with showing that $K_F: \cO(n+2-d)^{n-d+1}\oplus \cO(n+1-d)^{d-2}\to \cO(n+1)^n$ defines an injective map. We will show $K_F$ has maximum rank $n-1$ at all points $(s,t)\in \P^1$. Let $t=1$; the case $s=1$ is similar. We do it by Gauss-Jordan elimination. Move the first row to the last position, and use the $-t^{d-1}$ and the $-t^d$ along the diagonals to make their rows into zero. This shows that $K_F$ is equivalent to the matrix

\[ \begin{bsmallmatrix}
1 &   &  &   &   &   &   &   &   &   &   &   &   &  &   &  &  \\
 & \ddots &   &   &   &   &   &   &   &   &  &   &   &  &   &  &  \\
 &  & 1 &   &   &   &   &   &   &   &  &   &   &  &   &  &  \\
 &  &  &   &   &   &   &   &   &   &  &   &   &  &  1 &   &  \\
 &  &  &   &   &   &   &   &   &   &   &   &   &  &  & \ddots &  \\
 &  &  &   &   &   &   &   &   &   &  &   &   &  &   &  & 1 \\
 &  &  & s^{d-1}P_1 & 1+s^{d-1}P_2 & s+s^{d-1}P_3 & s^2+s^{d-1}P_4 & s^3+s^{d-1}P_5 & s^4+s^{d-1}P_6 & s^5+s^{d-1}P_7 & \cdots & s^{d-3}+s^{d-1}P_{d-1} & s^{d-2}+s^{d-1}P_{d} & s^{(d-1)M} &   &  & \\
 &  &  & s^{d-1} & 0 & 1  & s & s^2 & s^3 & s^4 &  & s^{d-4} & s^{d-3} & 0 &   &  & \\
 &  &  & \vdots & \vdots & \vdots & \vdots & \vdots & \vdots & \vdots & \cdots & \vdots & \vdots & \vdots &  &  & \\
 &  &  & s^2 & s^3 & s^4 & s^5 & s^6 & s^7 & s^8 &  & 0 &   & 0 &   &  & \\
 &  &  & s & s^2 & s^3 & s^4 & s^5 & s^6 & s^7 &  & s^{d-1} & 0 & 0 &   &  & \\
 &  &  & 1 & s & s^2 & s^3 & s^4 & s^5 & s^6 & \cdots & s^{d-2} & s^{d-1} & 0 &   &  &\\
 &  &   & s^{d-2} & s^{d-1} & 1 & s & s^2 & s^3 & s^4 &  & s^{d-5} & s^{d-4} & s^d-1 &   &  & 
\end{bsmallmatrix}. \]

Thus, we only need that the matrix
\[ \begin{bsmallmatrix}
s^{d-1}P_1 & 1+s^{d-1}P_2 & s+s^{d-1}P_3 & s^2+s^{d-1}P_4 & s^3+s^{d-1}P_5 & s^4+s^{d-1}P_6 & s^5+s^{d-1}P_7 & \cdots & s^{d-3}+s^{d-1}P_{d-1} & s^{d-2}+s^{d-1}P_{d} & s^{(d-1)M} \\
s^{d-1} & 0 &  1 & s & s^2 & s^3 & s^4 &  & s^{d-4} & s^{d-3} & 0 \\
\vdots & \vdots & \vdots & \vdots & \vdots & \vdots & \vdots & \cdots & \vdots & \vdots & \vdots \\
s^2 & s^3 & s^4 & s^5 & s^6 & s^7 & s^8 &  & 0 &   & 0 \\
s & s^2 & s^3 & s^4 & s^5 & s^6 & s^7 &  & s^{d-1} & 0 & 0 \\
1 & s & s^2 & s^3 & s^4 & s^5 & s^6 & \cdots & s^{d-2} & s^{d-1} & 0\\
s^{d-2} & s^{d-1} & 1 & s & s^2 & s^3 & s^4 &  & s^{d-5} & s^{d-4} & s^d-1
\end{bsmallmatrix} \]
has rank $d+1$. This can be shown by using the diagonal of $1$'s and induction. Therefore, we get $T_X|_C\cong \cO(n+2-d)^{n-d+1}\oplus \cO(n+1-d)^{d-2}$.
\end{proof}

By Corollary \ref{Corollary_Section_3} and Theorem \ref{theorem_higher_degree}, we have shown so far that a general Fano hypersurface $X\subset \P^n$ of degree $d\ge 3$ contains rational curves of degree $e$ with balanced restricted tangent bundle for every $\frac{n-1}{n+1-d} < e\le \max \{2d-2,n\}$. By Lemma \ref{gluing_balanced_curves}, we can glue a rational curve of degree $e_1$ with balanced restricted tangent bundle to a curve of degree $e_2$ with perfectly balanced restricted tangent bundle to obtain a degree $e_1+e_2$ rational curve with balanced restricted tangent bundle. This allows us to extend our result for all degrees $e$.

\begin{theorem}\label{theorem_higher_degree_curves}
    Let $X\subset \P^n$ be a general degree $d\ge 3$ Fano hypersurface. Then $X$ contains degree $e$ rational curves with balanced restricted tangent bundle for every degree $e > \frac{n-1}{n+1-d}$.
\end{theorem}
\begin{proof}
    First, notice that:
    \begin{itemize}[itemsep = 2pt]
        \item $2d-2\ge \left \lfloor\frac{n-1}{n+1-d} \right \rfloor + (n-1)$ for $\frac{n+3}{2}\le d\le n$, and
        \item $n\ge \left \lfloor\frac{n-1}{n+1-d} \right \rfloor + (n-1)$ for $d < \frac{n+3}{2}$.
    \end{itemize}
    Then, $\max \{2d-2,n\}\ge \left \lfloor \frac{n-1}{n+1-d} \right \rfloor + (n-1)$. Hence, $X$ contains rational curves with balanced restricted tangent bundle for every degree $\left \lfloor \frac{n-1}{n+1-d} \right \rfloor < e\le \left \lfloor \frac{n-1}{n+1-d}\right \rfloor + (n-1)$.
    
    Now, let $C_1$ be a rational curve in $X$ of degree $n-1$ with perfectly balanced restricted tangent bundle $T_X|_{C_1}\cong \cO(n+1-d)^{n-1}$ and $C_2$ be a rational curve of degree $e$ with balanced restricted tangent bundle $T_X|_{C_2}$. Since they are balanced, they are both free, then $C_1\cup C_2$ smooths into a degree $e+(n-1)$ rational curve $C$. By Lemma \ref{gluing_balanced_curves}, the general deformation of $C$ has balanced restricted tangent bundle. By gluing $m$ curves $C_1$, we get curves $C$ of degrees $e+m(n-1)$ for every integer $m\ge 0$. Since we have every $\frac{n-1}{n+1-d} < e\le \frac{n-1}{n+1-d} + (n-1)$, this gives us all degrees $e > \frac{n-1}{n+1-d}$.
\end{proof}

\printbibliography[title={References}]

\end{document}